\documentclass[10pt, a4paper, reqno]{amsart}
\usepackage{preambolo}
\usepackage{adjustbox}

\begin{document}


\setcounter{secnumdepth}{3}

\setcounter{tocdepth}{2}

\makeatletter
\def\author@andify{%
  \nxandlist {\unskip ,\penalty-1 \space\ignorespaces}%
    {\unskip {} \@@and~}%
    {\unskip \penalty-2 \space \@@and~}%
}
\makeatother

\title{\textbf{Classification of homogeneous almost complex $4$-manifolds with non-degenerate torsion bundle}}

\author[Cristina Bozzetti]{Cristina Bozzetti}

\author[Costantino Medori]{Costantino Medori}

\address{Costantino Medori: Dipartimento di Scienze Matematiche, Fisiche e Informatiche, Unità di Matematica e
Informatica, Università degli Studi di Parma, Parco Area delle Scienze 53/A, 43124, Parma, Italy}
\email{costantino.medori@unipr.it}

\author[Lorenzo Sillari]{Lorenzo Sillari}

\address{Lorenzo Sillari: Dipartimento di Scienze Matematiche, Fisiche e Informatiche, Unità di Matematica e Informatica, Università degli Studi di Parma, Parco Area delle Scienze 53/A, 43124, Parma, Italy} 
\email{lorenzo.sillari@unipr.it}

\maketitle

\begin{abstract}
\noindent \textsc{Abstract.} We investigate the local and global geometry of almost complex $4$-manifolds admitting non-degenerate torsion bundle. The rigidity of these structures forces a parallelizable $J$-adapted double cover, which imposes severe topological constraints on the underlying manifold. Exploiting this rigidity, we give a complete classification in the homogeneous setting. We show that such a manifold is diffeomorphic either to a $4$-dimensional Lie group carrying an almost complex structure with non-degenerate torsion bundle, or to a product $L(4,1)\times\mathbb R$ or $L(4,1)\times\mathbb T$, where $L(4,1)$ is a lens space. We also determine exactly which real $4$-dimensional Lie algebras admit such a structure. Constructively, we realize every admissible algebra by an explicit invariant structure, thereby closing the existence question in dimension $4$. We also relate these structures to certain Engel structures that we call
Nijenhuis--Engel, and answer the resulting existence questions in the homogeneous case.
\end{abstract}

\blfootnote{  \hspace{-0.55cm} 
{\scriptsize 2020 \textit{Mathematics Subject Classification}. Primary: 32Q60, 53C15; Secondary: 53C30, 57S20. \\ 
\textit{Keywords: automorphisms group, almost complex $4$-manifolds, Engel structures, Nijenhuis tensor, non-degenerate torsion bundle, transformation groups.}\\

\noindent The second and third authors are partially supported by GNSAGA of INdAM and by the Project PRIN 2022 “Real and Complex Manifolds: Geometry and Holomorphic Dynamics” (code 2022AP8HZ9)}
}

\section{Introduction}\label{sec:intro}

By the classical Newlander--Nirenberg Theorem, an almost complex structure $J$ on a smooth $2n$-manifold $M$ is locally equivalent to the standard complex structure on $\C^n$ if and only if its Nijenhuis tensor $N_J$ vanishes identically. In particular, all complex structures are locally equivalent. In the non-integrable setting, however, the situation is drastically different. While maximally non-integrable structures are topologically abundant in higher dimensions, satisfying an $h$-principle \cite{CPS22}, their local geometry is particularly rigid. Generic almost complex structures possess local invariants and typically admit no local symmetries \cite{Kob95}. This rigidity is more pronounced in dimension $4$, where continuous groups of local automorphisms are exceptionally rare unless the structure is integrable or strongly homogeneous.
\vspace{.2cm}

Modern approaches treat the image of the Nijenhuis tensor as a characteristic differential distribution $\mathcal{V} \coloneqq \Ima N_J$, relating it to the existence of foliations \cite{CGGH21, CGG24} or embeddings \cite{DG17}. When the distribution is regular, we call $\V$ the \textit{torsion bundle}, which must be preserved by symmetries and pseudoholomorphic functions. In dimension $4$, we call $\V$ \textit{non-degenerate} if the Lie brackets of its sections generate a distribution $\mathcal{E}$ of rank $3$. In particular, if $\V$ is non-degenerate, then $J$ is maximally non-integrable. We say that $\V$ is \textit{bracket-generating} if Lie brackets of sections of $\mathcal{E}$ span the whole $TM$. In such a case, the inclusion $\V \subset \mathcal{E} \subset TM$ defines an Engel structure on $M$.

The geometric rigidity of a non-degenerate torsion bundle allows the construction of canonical local $J$-adapted frames. Gluing these frames globally gives rise to a parallelizable \textit{$J$-adapted double cover} \cite{BM17}
\[
\pi: (F,J) \longrightarrow (M,J),
\]
which is a principal $\Z_2$-bundle over $M$. This construction imposes severe topological constraints on the underlying manifold: a compact almost complex $4$-manifold admitting a non-degenerate torsion bundle must have vanishing Euler characteristic $\chi(M)$ and signature $\sigma(M)$, see Lemma \ref{lemma:zero}. These local and global properties are intertwined with an increasing hierarchy of genericity conditions, summarized in Figure \ref{fig:genericity} .

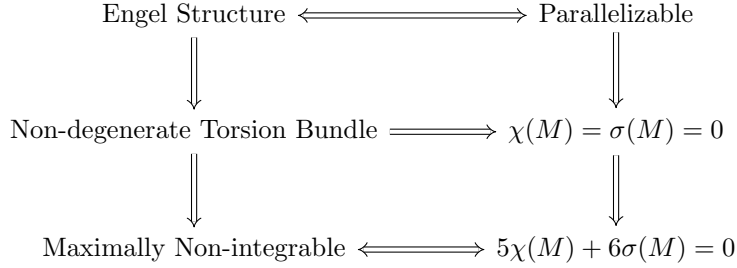
\begin{figure}[ht]
    \centering
    
    \begin{tikzcd}[row sep=large, column sep=large]
        \text{Engel Structure} \arrow[d, Rightarrow] \arrow[r, Leftrightarrow] & \text{Parallelizable} \arrow[d, Rightarrow] \\
        \text{Non-degenerate Torsion Bundle} \arrow[d, Rightarrow] \arrow[r, Rightarrow] & \chi (M) = \sigma (M) = 0 \arrow[d, Rightarrow]\\
        \text{Maximally Non-integrable} \arrow[r, Leftrightarrow] & 5\chi (M)+ 6\sigma (M) = 0
    \end{tikzcd}
    \caption{Genericity conditions and topological constraints for \textit{closed} almost complex $4$-manifolds.}\label{fig:genericity}
    
\end{figure}

In this paper, we exploit the rigidity of $J$-adapted double covers to obtain a complete classification of homogeneous almost complex $4$-manifolds with non-degenerate torsion bundle at the algebraic and topological level. First, using the classification of real $4$-dimensional Lie algebras, we determine exactly which of them carry an almost complex structure with non-degenerate torsion bundle: all but a short, explicit list of
exceptions, see Theorem \ref{thm:classification}.

\begin{thmx}
    The only $4$-dimensional Lie algebras not admitting an almost complex structure with non-degenerate torsion bundle are
    \begin{align*}
        4\g_1, \quad \g_{2,1} \oplus 2\g_1, \quad \g_{3,1} \oplus \g_1, \quad \g_{3,3} \oplus \g_1,
    \end{align*}
    the Lie algebras $\g_{4,2}^\alpha$ and $\g_{4,8}^\alpha$ for $\alpha=1$, and the Lie algebra $\g_{4,5}^{\alpha,\beta}$ for the choices of parameters $\alpha =1$, $\beta =1$, or $\alpha = \beta$.
\end{thmx}

For every admissible algebra we exhibit an explicit structure, and we compute whether its torsion bundle is non-degenerate or bracket-generating, see Examples \ref{ex:g'2} and \ref{ex:g'3}. These structures are left-invariant and descend to the corresponding Lie groups.

Then, by analyzing how the identity component $\Aut_0(M,J)$ lifts to the parallelizable cover $F$, we classify diffeomorphism types of homogeneous almost complex $4$-manifolds with non-degenerate torsion bundle. Our main differential result, Theorem \ref{thm:diffeotype}, is the following.

\begin{thmx}
Let $(M,J)$ be a connected homogeneous almost complex $4$-manifold with non-degenerate torsion bundle. Then $M$ is diffeomorphic to a $4$-dimensional Lie group admitting an almost complex structure with non-degenerate torsion bundle, or to the products $L(4,1) \times \R$, $L(4,1) \times \T$, where $L(4,1)$ is a lens space.
\end{thmx}

Theorems~A and~B together close the existence question in dimension $4$: every diffeomorphism type permitted by the topological classification is realized by an explicit invariant structure.
\vspace{.2cm}

Fix an almost complex structure $J$ on a parallelizable $4$-manifold and an Engel distribution $\mathcal{D} \subset \mathcal{E} \subset TM$. We call $\mathcal{D}$ a \emph{$J$-Engel} structure if $\mathcal{D}$ is $J$-invariant, and a \emph{Nijenhuis--Engel} structure if moreover $\mathcal{D}=\V$. In the integrable case, homogeneous $J$-Engel structures are classified in \cite{Zha18}, and the complex surfaces supporting one are determined in \cite{PP21}. This motivates asking which parallelizable $4$-manifolds admit a $J$-Engel, or a Nijenhuis--Engel, structure for some $J$. Our examples answer both questions for homogeneous almost complex $4$-manifolds.
\vspace{.2cm}

Our algebraic classification complements the differential one and provides a robust testing ground for several future projects on the understanding of canonical local forms and Engel distributions on homogeneous spaces. As a tool of independent interest, in Appendix \ref{sec:appendix} we classify the non-central subgroups of order $2$ of all connected $4$-dimensional Lie groups, a result essential in the proof of our topological classification.

\section{Torsion bundle of almost complex manifolds}\label{sec:prel}

In this section we recall basic facts on the torsion bundle of almost complex manifolds, giving special attention to the $4$-dimensional case. The main references for non-original results are \cite{BM17}, \cite{Kru04}, \cite{Kru98a} and \cite{Kru98b}.

\subsection{Nijenhuis tensor and torsion bundle.}

Let $(M,J)$ be a connected almost complex $2m$-manifold (not necessarily compact). The \textit{Nijenhuis tensor} of $J$ is defined by
\[
N_J(X,Y) \coloneqq [JX, JY] - J[JX, Y]- J[X,JY] - [X,Y],
\]
for $X$ and $Y$ vector fields on $M$. If $N_J=0$, we say that $J$ is \textit{integrable}. For all $p \in M$, the image of the Nijenhuis tensor is given by
\[
\V_p \coloneqq \Ima (N_J|_p) \subseteq T_p M,
\]
and it defines a distribution $\V$ in $TM$. If the distribution is regular, then $\V$ is a sub-bundle of $TM$, called the \textit{torsion bundle} of $(M,J)$. Since the Nijenhuis tensor satisfies the further symmetry
\[
N_J( JX, Y) = - J N_J(X,Y) = N_J(X,JY),
\]
the subspace $\V_p$ is $J$-invariant and its real dimension must be even. The \textit{(real) rank of the Nijenhuis tensor} (that differs by a factor of $2$ from the complex rank considered in \cite{CPS22}, \cite{Sil26} and \cite{ST23b}) is the real dimension
\[
\rk N_J|_p \coloneqq \dim_\R \V_p.
\]
The rank of $N_J$ measures how far $J$ is from being integrable, or, equivalently, the number of local independent pseudoholomorphic functions on $M$ in a neighborhood of a point, see \cite{Mus86} or \cite{HK12}. Complex structures have $N_J|_p =0$ for all $p \in M$ \cite{NN57}. In such a case, the torsion bundle identically vanishes. 

\subsection{Generic almost complex structures.}\label{sec:generic}

We focus on structures of constant rank, for which we have a well-defined torsion bundle. If $2m = 2$, every structure is necessarily integrable by dimensional reasons, and the torsion bundle vanishes. If $2m= 4$, the only possible values for the rank are $0$ and $2$. If $2m \ge 6$, the possible values for the rank of the torsion bundle are $0, 2, 4, \ldots,2m$. Structures with maximal rank are called \textit{maximally non-integrable}, and structures whose rank is never-vanishing are called \textit{everywhere non-integrable}. In dimension $4$ the two notions coincide, while in higher dimensions they differ. We point out that ``\textit{having rank at least $2k$ at every point}'', $k \in \N$, is a genericity condition for almost complex structures. In particular, maximally non-integrable and everywhere non-integrable structures are generic \cite{Sil26}.
\vspace{.2cm}

The Nijenhuis tensor is used by Kruglikov \cite{Kru98a} in the definition of another genericity condition. An almost complex structure is said to be \textit{of general position} if, for all $p \in M$ and for almost every $X_p \in T_pM$, the kernel of the map
\[
N_J(X_p, \cdot ) \colon T_pM \longrightarrow T_p M
\]
is two-dimensional. Observe that this is the minimal dimension for such a space, since the vectors $X_p$ and $JX_p$ always belong to $\ker N_J (X_p, \cdot)$. We prove that maximally non-integrable structures and structures of general position are closely related.

\begin{lemma}
    Let $(M,J)$ be an almost complex $2m$-manifold. If $2m = 4$, then $J$ is maximally non-integrable if and only if it is of general position. If $2m = 6$ and $J$ is maximally non-integrable, then it is of general position.
\end{lemma}
\begin{proof}
    Suppose that $J$ is maximally non-integrable. By contradiction, if $J$ is not of general position at $p \in M$, then there exists $X_p$ such that $\dim_\R \ker N_J(X_p,\cdot) >2$. In particular, there exists a vector $Y_p \notin \langle X_p, JX_p \rangle$ such that the space generated by $X_p$, $JX_p$, $Y_p$ and $JY_p$ is contained in $\ker N_J(X_p,\cdot)$. Complete the kernel to a basis of $T_pM$:
    \[
    \{ X, JX, Y, JY, V_1, JV_1, \ldots, V_{m-2}, JV_{m-2} \}_p.
    \]
    Then, the image of the Nijenhuis tensor is generated by vectors of the form $N_J(X,V_j)$, $N_J(Y,V_j)$, $N_J(V_j,V_k)$, and we can estimate its rank by
    \[
    \rk N_J|_p \le 2(m-2) + 2(m-2) + (m-2)(m-3) = m^2 -m -2.
    \]
    Recalling that the rank is an even number, we have that, for $m = 2,3$, it must be strictly smaller than $2m$, hence it is not maximal, contradicting the assumption of maximal non-integrability.
    For the converse implication, if $2m = 4$ and $J$ is of general position at $p \in M$, then $N_J|_p \neq 0$. If this was not the case, then we could take a basis $\{ X, JX, W, JW \}_p$ such that $N_J(X,W)_p =0$, giving $\dim_\R \ker N_J(X_p,\cdot)=4$ and leading to a contradiction.
\end{proof}

In dimension at least $10$, every almost complex structure on a \textbf{compact} manifold can be perturbed to be maximally non-integrable \cite{Sil26}, hence we can approximate structures of general position with maximally non-integrable structures.

\subsection{The \texorpdfstring{$4$}{}-dimensional case.}

Let $(M,J)$ be an almost complex $4$-manifold, and let $\V$ be the distribution defined by its Nijenhuis tensor. Suppose that $J$ is maximally non-integrable (or, equivalently, everywhere non-integrable). Hence, the distribution $\V$ is regular of constant rank $2$ and the torsion bundle $\V \subset TM$ is well-defined. Consider the distributions $\V^{(k)}$, $k \ge 1$, defined by
\[
\V^{(1)} \coloneqq \V, \qquad \V^{(k)} \coloneqq \V^{(k-1)} + [\V^{(1)}, \V^{(k-1)}] \quad \text{for $k \ge 2$}.
\]
This provides a series of inclusions
\[
\V^{(1)} \subseteq \V^{(2)} \subseteq \ldots \subseteq TM.
\]
The series stabilizes, i.e., one has $\V^{(k)} = \V^{(k-1)}$, after $0, 1$ or $2$ steps, depending on how the Lie bracket and the image of the Nijenhuis tensor interact with each other. The torsion bundle is said to be \textit{non-degenerate} if the series stabilizes after $1$ or $2$ steps, cf.\ \cite{KL10}. It is said to be \textit{bracket-generating} (or \textit{non-holonomic}) if the series stabilizes after $2$ steps, generating the whole tangent bundle. Having non-degenerate torsion bundle or having bracket-generating torsion bundle are, locally, genericity conditions for $J$, cf.\ \cite{Kru98a}.

\subsection{Non-degenerate torsion bundle.}\label{sec:cover} Almost complex $4$-manifolds with non-degenerate torsion bundle admit a local $J$-adapted frame. We recall here the main steps of the construction, and we refer to Section 6.1 in \cite{Kru98a}, Sections 1--3 in \cite{Kru04} or Sections 2 and 3 in \cite{BM17}, for more details. 

Fix $p \in M$ and let $X_p$ and $JX_p$ be generators of $\V_p$. Since the torsion bundle is non-degenerate, we have that $\V^{(2)}_p$ is $3$-dimensional and generated by the vectors $X_p, JX_p$ and $T_p \coloneqq [X,JX]_p$, with $T_p \notin \V_p$. Since $\V_p$ is $J$-invariant, we also have that $JT_p \notin \V_p$, which gives
\[
T_pM = \langle X_p, JX_p, T_p, JT_p \rangle.
\]
While the condition of non-degeneracy is independent of the choice of generators of $\V_p$, the local frame depends on it. However, up to a choice of sign, one can make a canonical choice of $X$ by fixing a vector field $A$ in the image of $N_J$ such that $A_p \neq 0$ and considering the map
\begin{align*}
\tau^A_p \colon \V_p &\longrightarrow \V_p \\
Y_p &\longmapsto N_J(Y,[A, JA])_p.
\end{align*}
Let $\pm \lambda_p$ be the eigenvalues of the map $\tau^A_p$. The corresponding eigenspaces $\V^\pm_p$ are independent of the choice of $A$ and they satisfy $J\V^+_p = \V^-_p$. The canonical choice of $X_p$ is made taking $X_p$ as the unique (up to sign) eigenvector in $\V^+_p$ for which $\lambda_p = 1$. We call the local frames built via this procedure \textit{$J$-adapted frames}, and we denote them by
\begin{align*}
f^+_p \colon \R^4 &\longrightarrow T_pM \\
\{ e_1, e_2, e_3, e_4 \} &\longmapsto \{ X_p, JX_p, T_p, JT_p \}
\end{align*}
and 
\begin{align*}
f^-_p \colon \R^4 &\longrightarrow T_pM \\
\{ e_1, e_2, e_3, e_4 \} &\longmapsto \{ -X_p, -JX_p, T_p, JT_p \},
\end{align*}
where the $e_j$, $j =1,\ldots, 4$, denote the standard basis of $\R^4$.

Pseudoholomorphic maps between almost complex manifolds send $J$-adapted frames into $J$-adapted frames. In particular, the only automorphisms of $(M,J)$ that fix a point $p$ are $\pm \Id$, see Corollary 3.7 in \cite{BM17}. This endows $(M,J)$ with a $\Z_2$-structure. 

The double cover $\pi \colon F \rightarrow M$, with fiber $\pi^{-1} (p) \coloneqq \{ f^+_p, f^-_p\}$ naturally arises as the principal bundle associated with the reduction of the structure group of $M$ to $\Z_2$. The $\Z_2$-structure on $M$ induced by $J$-adapted frames lifts to an $\{ e \}$-structure on $F$. Hence, $F$ is parallelizable and it inherits an almost complex structure obtained by lifting $J$, that we still denote by $J$. We call the cover
\[
\pi \colon (F,J) \longrightarrow (M,J)
\]
the \textit{$J$-adapted double cover} of $(M,J)$. By construction, $(F,J)$ is uniquely determined. In general, $F$ might not be connected, and we have the following.

\begin{lemma}\label{lemma:parallelizable}
  Let $(M,J)$ be a connected almost complex manifold and let $F$ be its $J$-adapted double cover. If $F$ has two connected components, then $M$ is parallelizable.  
\end{lemma}
\begin{proof}
    Fix $p \in M$. Then there exists an open neighborhood $U$ of $p$ such that $\pi^{-1} (U) = V \cup W$, with $V \cap W = \emptyset$, $f^+_p \in V$ and $f^-_p \in W$. Since $F$ has two connected components and $M$ is connected, the open sets $V$ and $W$ lie in different components of $F$. Then one of the components of $F$ is obtained by gluing the frames $f^+_p$ via the identity map and is mapped pseudo-biholomorphically onto $M$. If $\{ X_1, \ldots, X_4\}$ is a parallelism for that component, then $\{ \pi_* X_1, \ldots, \pi_* X_4 \}$ is a parallelism for $M$, and $F$ can be identified with the disjoint union of two copies of $M$.
\end{proof}

The existence of the double cover has strong constraints on the topology of $M$. Denote by $\chi(M)$ and $\sigma(M)$ the Euler characteristic and the signature of $M$, respectively.

\begin{lemma}\label{lemma:zero}
    Let $(M,J)$ be a compact almost complex $4$-manifold whose torsion bundle is non-degenerate. Then $\chi(M) = \sigma(M)=0$.
\end{lemma}
\begin{proof}
    Since $F$ is parallelizable, we have that $\chi(F) = \sigma (F) =0$. The Euler characteristic and signature of $F$ are related to those of $M$ by the equations
    \[
    \chi(F) = 2 \chi(M), \qquad \sigma(F) = 2 \sigma(M),
    \]
    since $F$ is an unbranched double cover of $M$, proving our claim.
\end{proof}

\subsection{Bracket-generating torsion bundle.} 
In contrast to what happens in dimension at least $6$, where every almost complex manifold admits everywhere non-integrable structures \cite{CPS22}, in dimension $4$ there are topological obstructions to the existence of everywhere (equivalently, maximally) non-integrable structures. By a result of Armstrong \cite{Arm96}, a \textbf{closed} almost complex $4$-manifold $M$ admits a maximally non-integrable structure if and only if $5\chi(M) + 6 \sigma(M) =0$. If we also impose non-degeneracy of the torsion bundle, we get that the manifold is parallelizable up to a double cover, as built in Section \ref{sec:cover}. Finally, we see that if the torsion bundle is bracket-generating, then the manifold has to be parallelizable.

\begin{lemma}
    Let $(M,J)$ be an almost complex $4$-manifold such that the torsion bundle of $J$ is bracket-generating. Then $M$ is parallelizable.
\end{lemma}
\begin{proof}
    Since the torsion bundle $\V$ is bracket-generating, there is a sequence of distributions
    \[
    \V \subset \mathcal{E} \subset TM,
    \]
    where $\mathcal{E} = \V + [\V, \V]$ is a sub-bundle of rank $3$ and $TM = \mathcal{E} + [\mathcal{E}, \mathcal{E}]$. Hence, the plane distribution given by $\mathcal{D} = \V$ defines an Engel structure on $M$. Existence of such a distribution implies that $M$ is parallelizable, cf.\ \cite{Vog09}.
\end{proof}
Existence of an Engel structure has the important consequence that the distribution $\V$ has a local normal form that allows us to choose a preferred orientation on a line distribution, called the \textit{characteristic line}, which is orientable \cite{Vog09}. This allows us to eliminate the sign ambiguity in the choice of $J$-adapted local frame and to glue the frame from local to global, endowing $M$ with a parallelism. In this setting, the Engel structure is intrinsically $J$-invariant, meaning $J\mathcal{D} = \mathcal{D}$. In the literature such a distribution is called a \textit{$J$-Engel structure}, see \cite{Zha18} or \cite{PP21}. Motivated by the study of $J$-Engel structures, we give the following.

\begin{definition}
An Engel structure $\mathcal{D}$ on a parallelizable $4$-manifold $M$ is called Nijenhuis--Engel if there exist $J$ on $M$ such that $\V = \mathcal{D}$.
\end{definition}

There are two natural questions that arise in this context.
\vspace{.2cm}

\textbf{Question:} which parallelizable $4$-manifolds admit a $J$-Engel structure for some almost complex structure $J$? Which ones admit a Nijenhuis--Engel structure?
\vspace{.2cm}

Our question will be extensively investigated with explicit examples in Section \ref{sec:algebras}. They provide a full answer in the case of homogeneous almost complex $4$-manifolds and invariant almost complex structures. Now we provide a benchmark example on $S^1 \times S^3$.

\begin{example}[\textbf{Nijenhuis--Engel distribution on $S^1 \times S^3$}]\label{ex:s1s3} The Lie algebra of $S^1 \times S^3$ is
\[
\mathfrak{so}(3) \oplus \R = \langle e_1, e_2, e_3, e_4 \rangle,
\]
with structure equations $[e_1, e_2] = e_3$, $[e_2, e_3] = e_1$, $[e_3, e_1] = e_2$. Consider the almost complex structure defined by $J e_1 = e_2$, $J e_3 = e_4 + e_1$, $J e_4 = -e_3 - e_2$. We easily see that $\V = \langle e_3, e_4 + e_1 \rangle$, $\V^{(2)} = \langle e_3, e_4 + e_1, e_2 \rangle$ and $\V^{(3)} = \langle e_3, e_4 + e_1, e_2, e_1 \rangle \cong \mathfrak{so}(3) \oplus \R$, providing a bracket-generating torsion bundle, hence a Nijenhuis--Engel distribution. \hfill $\blacksquare$    
\end{example}

\subsection{The automorphism group.}

Let $(M,J)$ be an almost complex $2m$-manifold and denote by $\Aut (M,J)$ the set of automorphisms of $(M,J)$ that are also pseudoholomorphic maps. Denote by $\Aut_0 (M,J)$ the connected component of the identity in $\Aut (M,J)$. An \textit{infinitesimal automorphism} of $(M,J)$ is a vector field $V$ such that
\[
[V, JX] = J[V,X]
\]
for every vector field $X$, i.e., such that $\mathcal{L}_V J = J \mathcal{L}_V$. The \textit{symmetry algebra} at $p \in M$ consists  of germs of infinitesimal automorphism at $p$ and it is denoted by $\aut_p (M,J)$, while its subset of elements that fix $p$ is the \textit{isotropy algebra}, and it is denoted by $\aut_p^0 (M,J)$.

For an in-depth study of the pseudo-group of local symmetries in arbitrary dimension, we refer to \cite{BKW63}, \cite{Kob95}, \cite{KW16}, \cite{KW17} and \cite{Kru14}. Here we confine ourselves to the $4$-dimensional case. We recall the following result, contained in Theorem C in \cite{Kru14} and in Theorem 3.10 in \cite{BM17}, see also \cite{KL10}.

\begin{theorem}\label{thm:auto:group}
    Let $(M,J)$ be an almost complex $4$-manifold whose torsion bundle is non-degenerate. Then $G= \Aut_0(M,J)$ is a Lie group of dimension at most $4$, and $G$ is $4$-dimensional if and only if $M$ is a homogeneous space for $G$ and $J$ is an invariant almost complex structure.
\end{theorem}

We invite the reader to compare this result with Theorem 3.2 and Corollary 4.2 in \cite{Kob95}, where it is underlined that almost complex structures with many symmetries are non-generic, while generic structures have few symmetries, if none at all. 

\section{Structures with non-degenerate torsion bundle on Lie algebras}\label{sec:algebras}

In this section we classify $4$-dimensional Lie algebras admitting an almost complex structure with non-degenerate torsion bundle, together with several explicit examples.
\vspace{.2cm}

Let $\g$ be a $4$-dimensional Lie algebra. We denote by $\g^{(k)}$, $k \ge 1$, the derived series of $\g$. For $k = 1$, we denote $\g^{(1)} = \g'$. Tables 2 and 3 in \cite{BR16} give a complete list of $4$-dimensional Lie algebras and their commutators. Since we adopt the same notation, we reproduce the list in Table \ref{tab:list} for convenience of the reader.
\vspace{.2cm}

The following theorem summarizes the results of this section.

\begin{theorem}\label{thm:classification}
    The only $4$-dimensional Lie algebras not admitting an almost complex structure with non-degenerate torsion bundle are
    \begin{align*}
        4\g_1, \quad \g_{2,1} \oplus 2\g_1, \quad \g_{3,1} \oplus \g_1, \quad \g_{3,3} \oplus \g_1,
    \end{align*}
    the Lie algebras $\g_{4,2}^\alpha$ and $\g_{4,8}^\alpha$ for $\alpha=1$, and the Lie algebra $\g_{4,5}^{\alpha,\beta}$ for the choices of parameters $\alpha =1$, $\beta =1$, or $\alpha = \beta$.
\end{theorem}
\begin{proof}
    By Proposition 4.1 in \cite{BM17}, if $J$ is an almost complex structure with non-degenerate torsion bundle on $\g$, then $\dim_\R \g' \in \{2, 3\}$. This allows us to conclude that $4\g_1$, $\g_{2,1} \oplus 2\g_1$, and $\g_{3,1} \oplus \g_1$ do not admit such a structure. The remaining cases need \textit{ad hoc} proofs, which are contained in Lemmas \ref{lemma:g33} and \ref{lemma:g42}. To show that every $4$-dimensional Lie algebra not listed in the statement of the theorem admits an almost complex structure with non-degenerate torsion bundle, we provide explicit examples, see Examples \ref{ex:g'2} and \ref{ex:g'3}.
\end{proof}

In the following lemmas, we show that there are Lie algebras on which every almost complex structure has degenerate torsion bundle.

\begin{lemma}\label{lemma:g33}
    Every almost complex structure on $\g_{3,3} \oplus \g_1$ has degenerate torsion bundle.
\end{lemma}
\begin{proof}
    
    Let $J$ be an arbitrary almost complex structure on $\g_{3,3} \oplus \g_1$, and set $J e_j = J^k_j e_k$, where we are summing over repeated indices. Explicit computations show that the Nijenhuis tensor is given by
    \begin{align}
        & N_J(e_1,e_2) = (J_2^3 J_1^4 - J_1^3 J_2^4) e_4, \label{Nj:12}\\
        & N_J(e_1,e_3) = (J_1^4 J_3^3 - J_1^3 J_3^4) e_4 - J_1^4 Je_4, \label{Nj:13} \\
        & N_J(e_1,e_4) = (J_1^4 J_4^3 - J_1^3 J_4^4) e_4 + J_1^3 Je_4, \\
        & N_J(e_2,e_3) = (J_2^4 J_3^3 - J_2^3 J_3^4) e_4 - J_2^4 Je_4, \\
        & N_J(e_2,e_4) = (J_2^4 J_4^3 - J_2^3 J_4^4) e_4 + J_2^3 Je_4, \\
        & N_J(e_3,e_4) = (1 + J_3^4 J_4^3 - J_3^3 J_4^4) e_4 + (J_3^3 + J_4^4) Je_4.\label{Nj:34}
    \end{align}
    The computations are straightforward and, for the sake of the reader, we give more details on how to obtain equation \eqref{Nj:13}. We have that
    \begin{align*}
        N_J(e_1,e_3) &= [Je_1, Je_3] - J([Je_1,e_3] + [e_1, Je_3]) - [e_1, e_3]\\
        &= J_1^l J_3^k [e_l, e_k] -J ( J_1^l [e_l,e_3] + J_3^l[e_1, e_l]) +e_1 \\
        &= e_1 (1 + J_1^3 J_3^1 - J_1^1 J_3^3) + e_2 (J_1^3 J_3^2 - J_1^2 J_3^3)\\
        &+ J_1^1 Je_1 + J_3^3 Je_1 + J_1^2 Je_2 \\
        &= e_1 (1 + J_1^3 J_3^1 + (J_1^1)^2 + J_1^2 J_2^1) \\
        &+ e_2 (J_1^3 J_3^2 + J_1^1 J_1^2 + J_1^2 J_2^2) \\
        &+ e_3 (J_1^3 J_3^3 + J_1^1 J_1^3 + J_1^2 J_2^3) \\
        &+ e_4 (J_1^4 J_3^3 + J_1^1 J_1^4 + J_1^2 J_2^4). 
    \end{align*}
    Imposing the condition $J^2 = - \Id$, we have, for instance, that $1 + (J_1^1)^2 + J_1^2 J_2^1 + J_1^3 J_3^1  = - J_1^4 J_4^1$. Using similar equations, the expression of the Nijenhuis tensor can be simplified to 
    \[
    N_J(e_1,e_3) = (J_1^4 J_3^3 - J_1^3 J_3^4) e_4 - J_1^4 Je_4.
    \]
    The fact that $J$ must be degenerate follows from equations \eqref{Nj:12}--\eqref{Nj:34}. Indeed, if every Nijenhuis tensor vanishes, then $J$ is integrable. If at least one of them is non-vanishing, then $e_4 \in \V$ and $\V =\langle e_4, Je_4 \rangle$. Hence the torsion bundle is necessarily degenerate since either $J$ is integrable, or $\V =\langle e_4, Je_4 \rangle$ and $[\V, \V] =0$.
\end{proof}

\begin{lemma}\label{lemma:g42}\hfill
    \begin{itemize}
        \item [(i)] Every almost complex structure on $\g_{4,2}^\alpha$, for $\alpha = 1$, and on $\g_{4,8}^\alpha$, for $\alpha =1$, has degenerate torsion bundle. 

        \item [(ii)] Every almost complex structure on $\g_{4,5}^{\alpha, \beta}$ has degenerate torsion bundle if the parameters satisfy at least one of the conditions $\alpha =1$, $\beta =1$ or $\alpha = \beta$.
    \end{itemize}
\end{lemma}
\begin{proof}
    The idea of the proof is the same as that of Lemma \ref{lemma:g33}. We start with $\g_{4,2}^1$. We show that for every almost complex structure $J$, we have that $e_2 \in \V$. This implies that $\V = \langle e_2, Je_2 \rangle$ and $[\V, \V] \subseteq \langle e_2 \rangle \subseteq \V$, giving degeneracy of the torsion bundle. We explicitly have that
    \begin{align*}
        & N_J(e_1,e_2) = (J_2^4 J_1^3 - J_1^4 J_2^3) e_2,\\
        & N_J(e_1,e_3) = (J_1^3 J_3^4 - J_1^4 J_3^3) e_2 + J_1^4 Je_2, \\
        & N_J(e_1,e_4) = (J_1^3 J_4^4 - J_1^4 J_4^3) e_2 - J_1^3 Je_2, \\
        & N_J(e_2,e_3) = (J_2^3 J_3^4 - J_2^4 J_3^3) e_2 + J_2^4 Je_2, \\
        & N_J(e_2,e_4) = (J_2^3 J_4^4 - J_2^4 J_4^3) e_2 - J_2^3 Je_2, \\
        & N_J(e_3,e_4) = (J_3^3 J_4^4 - 1) e_2 - (J_3^3 + J_4^4) Je_2.
    \end{align*}
    
    For $\g_{4,8}^1$, explicit computations show that for all indices $j < k$, we have that $N_J(e_j, e_k) \in \langle e_1, Je_1 \rangle$. Hence, either the structure $J$ is integrable or $\V = \langle e_1, Je_1 \rangle$. In both cases, the torsion bundle is degenerate.
    
    The proof for the Lie algebra $\g_{4,5}^{\alpha,\beta}$ follows the same ideas, with more complicated computations. By the constraints on the parameters $-1 \le \alpha \le \beta \le 1$, see Table \ref{tab:list}, we can restrict ourselves to the cases $\beta =1$ and $\alpha = \beta$. For the sake of completeness, we give the expression for the Nijenhuis tensors.


    \textbf{Case $\beta =1$:}
    \begin{align*}
        & N_J(e_1,e_2) = (\alpha-1)(J_1^3 J_2^4 - J_1^4 J_2^3) e_3, \\
        & N_J(e_1,e_3) = (\alpha-1)(J_1^3 J_3^4 - J_1^4 J_3^3)e_3 + (\alpha-1)J_1^4 Je_3, \\
        & N_J(e_1,e_4) = (\alpha-1)(J_1^3 J_4^4 - J_1^4 J_4^3) e_3 - (\alpha-1) J_1^3 Je_3, \\
        & N_J(e_2,e_3) = (\alpha-1)(J_2^3 J_3^4 - J_2^4 J_3^3) e_3  + (\alpha-1)J_2^4 Je_3, \\
        & N_J(e_2,e_4) =  (\alpha-1)(J_2^3 J_4^4 - J_2^4 J_4^3) e_3 - (\alpha-1) J_2^3 Je_3, \\
        & N_J(e_3,e_4) =  (\alpha-1)(J_3^3 J_4^4 - J_3^4 J_4^3  - 1) e_3 - (\alpha-1) (J_3^3 + J_4^4) Je_3
    \end{align*}

    \textbf{Case $\alpha = \beta $:}
    \begin{align*}
        & N_J(e_1,e_2) = (\beta-1)(J_1^4 J_2^1 - J_1^1 J_2^4) e_1 + (\beta-1)J_2^4 Je_1, \\
        & N_J(e_1,e_3) = (\beta-1)(J_1^4 J_3^1 - J_1^1 J_3^4) e_1 + (\beta-1)J_3^4 Je_1, \\
        & N_J(e_1,e_4) = (\beta-1)(1 - J_1^1 J_4^4) e_1 + (\beta-1)(J_1^1 + J_4^4) Je_1, \\
        & N_J(e_2,e_3) = (\beta-1)(J_2^4 J_3^1 - J_2^1 J_3^4) e_1, \\
        & N_J(e_2,e_4) = (\beta-1)(J_2^4 J_4^1 - J_2^1 J_4^4) e_1 + (\beta-1) J_2^1 J e_1, \\
        & N_J(e_3,e_4) = (\beta-1)(J_3^4 J_4^1 - J_3^1 J_4^4) e_1 +(\beta-1) J_3^1 Je_1.
    \end{align*}
    In each case, we have that either $J$ is integrable or the torsion bundle is degenerate.
\end{proof}

We provide an alternative abstract proof for some specific cases.

\begin{lemma}
     Let $\g$ be a $4$-dimensional Lie algebra with $\dim_\R \g' = 3$ and $\g^{(2)} = \{ 0 \}$. Suppose that there exists a $2$-plane $\pi \subseteq \g'$ and a vector $Z_0 \notin \g'$ such that $\mathrm{ad}_{Z_0}|_{\pi} = \Id$. Then every almost complex structure on $\g$ has degenerate torsion-bundle.

     In particular, the plane $\pi$ and the vector $Z_0$ exist for $\g_{4,2}^\alpha$, if $\alpha =1$, and for $\g_{4,5}^{\alpha, \beta}$, if $\alpha =1$, $\beta =1$, or $\alpha = \beta$.
\end{lemma}
\begin{proof}
     Suppose by contradiction that $J$ is such a structure, and let $Y_0$ a generator of $\V \cap \g'$, so that $\V = \langle Y_0, J Y_0 \rangle$. Since $\g^{(2)} = \{ 0 \}$, we cannot have $\V \subset \g'$. Hence $JY_0 = \alpha Z_0 + X_0$, with $X_0 \in \g'$ and $\alpha \neq 0$. Since $\V$ is non-degenerate, then $Y_0 \notin \pi$ because if it was the case, we would have $T \coloneqq [JY_0, Y_0] = \alpha [Z_0, Y_0] = Y_0 \in \V$. Hence we get the direct sum $\g' = \pi \oplus \langle Y_0 \rangle$. Up to a rescaling of $Y_0$, we can assume $\alpha = 1$, hence we can write $T = [Z_0, Y_0] = \gamma Y_0 + w'$, with $w' \in \pi$ and $w' \neq 0$ since the torsion bundle is non-degenerate. If we set $J w' = \epsilon JY_0 + \delta Y_0 + w''$, with $w'' \in \pi$, $w''\neq 0$, we can compute
     \begin{align*}
     N_J (Y_0, T) &= N_J (Y_0, w')\\
     &= [JY_0, \epsilon JY_0 + \delta Y_0 + w''] - J( [JY_0, w'] + [Y_0, \epsilon JY_0 + \delta Y_0 + w'']) \\
     &= \delta [JY_0, Y_0] + w'' - Jw' + \epsilon J [JY_0, Y_0]\\
     &= \delta T + \epsilon J T - \delta Y_0 - \epsilon JY_0.
     \end{align*}
     Since $\{ Y_0, JY_0, T, JT \}$ is a basis of $\g$ and $\{Y_0, JY_0 \}$ is a basis of $\V$, it must be $\epsilon = \delta =0$. This is absurd because it would force $N_J$ to identically vanish. 
\end{proof}

We now give an example of almost complex structures with non-degenerate torsion bundle on each $4$-dimensional Lie algebra that admits one. We do that by means of $J$-bracket-adapted frames, which take into account the simultaneous action of the almost complex structure and of the Lie bracket, or of almost complex structures whose torsion bundle is bracket-generating.

\begin{definition}
    Let $\{X, J X, T, J T \}$ be a $J$-adapted frame for $\g$. We say that the frame is \textit{$J$-bracket-adapted} if $X \in \V \cap \g'$. 
\end{definition}

If $\dim_\R \g' = 3$, the construction of $J$-bracket-adapted frames is straightforward. 

\begin{lemma}\label{lemma:dim3:adapted}
    Let $\g$ be a $4$-dimensional Lie algebra with $\dim_\R \g' = 3$, and let $J$ be an almost complex structure on $\g$ with non-degenerate torsion bundle. Then $\g$ admits a $J$-bracket-adapted frame.
\end{lemma}
\begin{proof}
    Let $\langle X, JX, T, JT \rangle$ be a $J$-adapted frame on $\g$, with $\V = \langle X, JX \rangle$. Since $\dim_\R \g' = 3$ and $\dim_\R \V =2$, we have that $\dim \V \cap \g' \in \{ 1, 2\}$. 
    
    If $\dim \V \cap \g' = 2$, then $\V \subset \g'$, $[\V, \V] \notin \V$ by assumption, and any $J$-adapted frame on $\g$ is also $J$-bracket adapted. By a general fact on $J$-adapted frames, we have uniqueness up to a sign. Choosing a different element of $\V$ as the first element of the frame, say $\hat X = aX + bJX$, we see that $[\hat X, J \hat X] = \hat T = (a^2 +b^2) T = [X, JX]$ and that $\hat T \in \g^{(2)}$ is determined up to a positive constant.
    
    If $\dim \V \cap \g' = 1$, let $\hat X$ be a generator of $\V \cap \g'$. Then $\V = \langle \hat X, J \hat X \rangle$, with $J \hat X \notin \g'$. By setting $\hat T = [\hat X, J \hat X]$, the frame $\{ \hat X, J\hat X, \hat T, J \hat T \}$ is $J$-bracket-adapted. Similarly as the previous case, we have uniqueness up to a positive constant.
\end{proof}

From the first part of the proof of Lemma \ref{lemma:dim3:adapted}, we deduce the following

\begin{cor}
    If $\g$ is a Lie algebra with $\dim_\R \g' =3$ and $\g^{(2)} = \{ 0 \}$, and $J$ is an almost complex structure on $\g$ with non-degenerate torsion bundle, then $\dim \V \cap \g' = 1$.
\end{cor}

For a fixed Lie algebra $\g$, one can use the dimension of $\V \cap \g'$ to distinguish different almost complex structures. For a generic $J$, such a dimension will be the minimum possible, that is $0$ if $\dim_\R \g' = 2$ and $1$ if $\dim_\R \g' = 3$. Note that, in the first case, the condition $\dim_\R \V \cap \g'=0$ is necessary for $J$ to be bracket-generating.

\begin{table}[ht]
\centering
\caption{Dimension analysis of $\V \cap \g'$.}
\label{table:intersection}
\renewcommand{\arraystretch}{1.5}
\makebox[\textwidth][c]{
\begin{tabular}{c|c|c}
& $\dim_\R \g' = 2$ & $\dim_\R \g' = 3$  \\ \hline
$\dim_\R \V \cap \g' = 0$ & \makecell{possibly non-degenerate \\ possibly bracket-generating} & $\V = \{ 0 \}$, integrable \\ \hline
$\dim_\R \V \cap \g' = 1$ & \makecell{possibly non-degenerate \\ not bracket-generating} & \makecell{possibly non-degenerate \\ possibly bracket-generating} \\ \hline
$\dim_\R \V \cap \g' = 2$ & only degenerate & \makecell{possibly non-degenerate $( \Rightarrow \g^{(2)} \neq \{ 0 \} )$ \\ not bracket-generating}
\end{tabular}
}
\end{table}

We show that in certain cases the condition $\V \cap \g'=\{0\}$ is also sufficient to have bracket-generating torsion bundle.

\begin{lemma}\label{lemma:existence}
    Let $J$ be an almost complex structure on the Lie algebras $\g_{3,5}^\alpha \oplus \g_1$, for $\alpha \ge 0$, $\g_{4,1}$, or $\g_{4,10}$. If the torsion bundle $\V$ is non-degenerate and satisfies $\V \cap \g'=\{0\}$ then $\V$ is bracket-generating.
\end{lemma}
\begin{proof}
    The assumption $\V \cap \g' = \{0 \}$ gives that $\dim_\R \g' = 2$ and $\g = \V \oplus \g'$. Since $\g' = \langle e_1, e_2 \rangle$, we can find a basis of $\V$ of the form $\{ e_3 + v, e_4 + w\}$ with $v, w \in \g'$. In general, this basis will not satisfy $J(e_3+v) = e_4 +w$. Suppose that $\g \neq \g_{4,1}$. By looking at Table \ref{tab:list}, we see that $\g'$ is spanned by the action of $[e_3, \cdot]$ and $[e_4, \cdot]$. Represent by a matrix $A$ the action of $ad_{e_3}|_{\g'}$. Similarly, represent by $B$ the action of $ad_{e_4}|_{\g'}$. Since $\V$ is non-degenerate, the vector $T = Aw - Bv \in \g'$ is non-zero. By taking iterated brackets, we see that $\V$ is bracket-generating if and only if $AT$ and $BT$ are independent from $T$. This is guaranteed if $A$ and $B$ have no common real eigenvectors, which is always the case for the Lie algebras appearing in the statement of the Lemma. The case of $\g_{4,1}$ yields the same conclusion after observing that the structure equations force $T = e_2 + A e_1$, with $A \in \R$, that corresponds to a bracket-generating structure.
\end{proof}

\begin{remark}
    The results of Lemma \ref{lemma:existence} generalize to Lie algebras with non-degenerate torsion bundles satisfying $\V \cap \g' = \{0\}$. Specifically, $\V$ is bracket-generating precisely when $T$ is not a common eigenvector of $A$ and $B$. Verifying this condition needs a deeper analysis of the local behavior of the torsion bundle within $\g$.
\end{remark}

\begin{example}[Explicit structures on Lie algebras with $\dim_\R \g' = 2$]\label{ex:g'2}
In this example we provide an almost complex structure with a $J$-bracket-adapted frame for each $4$-dimensional Lie algebra with $2$-dimensional derived Lie algebra. In such a case, we have that $\dim_\R \V \cap \g' = 1$. Similarly, we provide an almost complex structure whose torsion bundle is bracket-generating, and $\V \cap \g' = \{ 0 \}$. Note that when $\dim_\R \g' =2$, structures with bracket-generating torsion bundle do not admit $J$-bracket-adapted frames. 
\vspace{.2cm}

$\bullet \, 2 \g_{2,1}$: the structure equations are
\[
[e_1, e_2] = e_1, \quad [e_3, e_4] = e_3.
\]
Consider the structure defined by $Je_1 = e_4$, $Je_2 = e_3$. Then $\V \cap \g' = \langle e_1 + e_3 \rangle$ and we get a $J$-bracket-adapted frame 
\[
\langle e_1 + e_3, e_4 - e_2, e_3 - e_1, - e_2 - e_4 \rangle.
\]
For the structure defined by $Je_1 = e_1 + e_3 + e_4$, $Je_4 = -2 e_1 + e_2 - e_3 - e_4$, $Je_2 = e_3$, we have that $\V \cap \g' = \{ 0 \}$. In this case, the structure is bracket-generating since 
\[
\V = \langle 2 e_1 + 2 e_3 + e_4, e_2 - e_3 - e_4 \rangle.
\]

$\bullet \, \g_{3,2} \oplus \g_1$: the structure equations are
\[
[e_2, e_3] = e_1 - e_2, \quad [e_3, e_1] = e_1.
\]
Consider the structure defined by $Je_1 = e_3$, $Je_2 = e_4$. Then $\V \cap \g' = \langle e_1 - e_2 \rangle$ and we get a $J$-bracket-adapted frame 
\[
\langle e_1 - e_2, e_3 - e_4, -2e_1 + e_2, -2e_3 + e_4 \rangle.
\]
For the structure defined by $Je_2 = e_1 + e_2 + e_4$, $Je_4 = - e_1 -2 e_2 - e_3 - e_4$, $Je_1 = e_3$, we have that $\V \cap \g' = \{ 0 \}$. In this case, the structure is bracket-generating since 
\[
\V = \langle e_1 - e_3 + e_4, 2 e_2 + e_4 \rangle.
\]

$\bullet \, \g_{3,4}^0 \oplus \g_1$: the structure equations are
\[
[e_2, e_3] = e_1, \quad [e_3, e_1] = -e_2.
\]
Consider the structure defined by $Je_1 = e_4$, $Je_2 = e_3$. Then $\V \cap \g' = \langle e_2 \rangle$ and we get a $J$-bracket-adapted frame 
\[
\langle e_2, e_3, e_1, e_4 \rangle.
\]
For the structure defined by $Je_1 = e_2 + e_4$, $Je_4 = - e_1 - e_3$, $Je_2 = e_3$, we have that $\V \cap \g' = \{ 0 \}$. In this case, the structure is bracket-generating since 
\[
\V = \langle e_1 - e_3, 2 e_2 + e_4 \rangle.
\]

$\bullet \, \g_{3,4}^\alpha \oplus \g_1$: the structure equations are
\[
[e_2, e_3] = e_1 - \alpha e_2, \quad [e_3, e_1] = \alpha e_1 - e_2 \quad \text{with $\alpha >0$ and $\alpha \neq 1$.}
\]
Consider the structure defined by $Je_1 = e_4$, $Je_2 = e_3$. Then $\V \cap \g' = \langle \alpha e_1 - e_2 \rangle$ and we get a $J$-bracket-adapted frame 
\[
\langle \alpha e_1 - e_2, \alpha e_4 - e_3, (\alpha^2 + 1) e_1 - 2\alpha e_2, (\alpha^2 + 1)e_4 -2\alpha e_3 \rangle.
\]
For the structure defined by $Je_1 = e_2 + e_4$, $Je_4 = - e_1 - e_3$, $Je_2 = e_3$, we have that $\V \cap \g' = \{ 0 \}$. In this case, the structure is bracket-generating since 
\[
\V = \langle e_1 - e_3 + \alpha e_4, - \alpha e_1 + 2 e_2 - \alpha e_3 + e_4 \rangle.
\]

$\bullet \, \g_{3,5}^0 \oplus \g_1$: the structure equations are
\[
[e_2, e_3] = e_1, \quad [e_3, e_1] = e_2.
\]
Consider the structure defined by $Je_1 = e_4$, $Je_2 = e_3$. Then $\V \cap \g' = \langle e_2 \rangle$ and we get a $J$-bracket-adapted frame 
\[
\langle e_2, e_3, e_1, e_4 \rangle.
\]
For the structure defined by $Je_1 = e_2 + e_3 + e_4$, $Je_4 = - e_1 + e_2 - e_3 $, $Je_2 = e_3$, we have that $\V \cap \g' = \{ 0 \}$. In this case, the structure is bracket-generating since 
\[
\V = \langle e_1 - e_3 - e_4, e_1 + e_2 + 2 e_3 + e_4 \rangle.
\]

$\bullet \, \g_{3,5}^\alpha \oplus \g_1$: the structure equations are
\[
[e_2, e_3] = e_1 - \alpha e_2, \quad [e_3, e_1] = \alpha e_1 + e_2 \quad \text{with $\alpha >0$}.
\]
Consider the structure defined by $Je_1 = e_4$, $Je_2 = e_3$. Then $\V \cap \g' = \langle \alpha e_1 + e_2 \rangle$ and we get a $J$-bracket-adapted frame 
\[
\langle \alpha e_1 + e_2, \alpha e_4 + e_3, (1-\alpha^2)e_1 - 2\alpha e_2, (1-\alpha^2)e_4 - 2\alpha e_3 \rangle.
\]
For the structure defined by $Je_1 = e_2 + e_3 + e_4$, $Je_4 = - e_1 + e_2 - e_3$, $Je_2 = e_3$, we have that $\V \cap \g' = \{ 0 \}$. In this case, the structure is bracket-generating since 
\[
\V = \langle e_1 + e_2 + 2 e_3 + (1+\alpha) e_4, - (1+\alpha)e_1 + \alpha e_2 + (1-\alpha) e_3 + e_4 \rangle.
\]

$\bullet \, \g_{4,1}$: the structure equations are
\[
[e_2, e_4] = e_1, \quad [e_3, e_4] = e_2.
\]
Consider the structure defined by $Je_1 = e_3$, $Je_2 = e_4$. Then $\V \cap \g' = \langle e_2 \rangle$ and we get a $J$-bracket-adapted frame 
\[
\langle e_2, e_4, e_1, e_3 \rangle.
\]
For the structure defined by $Je_1 = e_4$, $Je_2 = e_3$, we have that $\V \cap \g' = \{ 0 \}$. In this case, the structure is bracket-generating since 
\[
\V = \langle e_1 +  e_3, e_4 -e_2\rangle.
\]

\vspace{.05cm}

$\bullet \, \g_{4,3}$: the structure equations are
\[
[e_1, e_4] = e_1, \quad [e_3, e_4] = e_2.
\]
Consider the structure defined by $Je_1 = e_4$, $Je_2 = e_1 + e_2 + e_3$, $Je_3 = -e_1 -2 e_2 - e_3 - e_4$. Then $\V \cap \g' = \{ e_1 + e_2\}$, and we get a $J$-bracket-adapted frame 
\[
\langle e_1 + e_2, e_1 + e_2 + e_3 + e_4, e_1, e_4 \rangle. 
\]
For the structure defined by $Je_1 = -e_1 + e_2 + e_3 + 2e_4$, $Je_2 = - 2 e_1 + e_2 + 3 e_3 + e_4$, $Je_3 = e_4$, we have that $\V \cap \g' = \{ 0 \}$. In this case, the structure is bracket-generating since 
\[
\V = \langle 4 e_1 -  2 e_2 - 5 e_3, 2 e_2 - 2 e_3 + e_4\rangle.
\]

$\bullet \, \g^0_{4,8}$: the structure equations are
\[
[e_1, e_4] = e_1, \quad [e_2, e_4] = e_2, \quad [e_2, e_3] = e_1.
\]
Consider the structure defined by $Je_3 = e_4$, $Je_1 = e_1 + e_2 + e_3 + e_4$, $J e_2 = -2 e_1 -e_2 -2 e_4$. Then $\V \cap \g' = \{ 0 \}$ and we get a bracket-generating torsion bundle
\[
\V = \langle e_1 + e_3 - e_4, e_1 + e_2 + 2 e_3 + 2e_4 \rangle.
\]
For the construction of a $J$-bracket-adapted frame, see the general case $\g_{4,8}^\alpha$ in Example \ref{ex:g'3}.

$\bullet \, \g_{4,10}$: the structure equations are
\[
[e_1, e_3] = e_1, \quad [e_2, e_3] = e_2, \quad [e_1, e_4] = - e_2, \quad [e_2, e_4] = e_1.
\]
Consider the structure defined by $Je_1 = e_3$, $Je_2 = e_4$. Then $\V \cap \g' = \langle e_2 \rangle$ and we get a $J$-bracket-adapted frame 
\[
\langle e_2, e_4, e_1, e_3 \rangle.
\]
For the structure defined by $Je_1 = e_3$, $Je_2 = e_1 + e_4$, $Je_4 = -e_2 - e_3$, we have that $\V \cap \g' = \{ 0 \}$. In this case, the structure is bracket-generating since 
\[
\V = \langle e_1 + 2 e_4, 2 e_2 + e_3\rangle.
\]

\hfill $\blacksquare$ 
\end{example}

\begin{example}[Explicit structures on Lie algebras with $\dim_\R \g' = 3$]\label{ex:g'3}
In this example we provide an almost complex structure with a $J$-bracket-adapted frame for each $4$-dimensional Lie algebra with $3$-dimensional derived Lie algebra.

$\bullet \, \g_{3,6} \oplus \g_1$: the structure equations are
\[
[e_2, e_3] = e_1, \quad [e_3, e_1] = e_2, \quad [e_1, e_2] = -e_3.
\]
Consider the structure defined by $Je_1 = e_4$, $Je_2 = e_3$. Then $\V \cap \g' = \langle e_2, e_3 \rangle$ and we get a $J$-bracket-adapted frame 
\[
\langle e_2, e_3, e_1, e_4 \rangle.
\]
For the structure defined by $Je_1 = e_4$, $Je_2 = e_3 + e_4$, $Je_3 = e_1 - e_2$, we have that $\V \cap \g' = \{ e_1 -2 e_2 \}$. In this case, the structure is bracket-generating since 
\[
\V = \langle e_1 - 2 e_2, 2 e_3 + e_4\rangle.
\]

$\bullet \, \g_{3,7} \oplus \g_1$: the structure equations are
\[
[e_2, e_3] = e_1, \quad [e_3, e_1] = e_2, \quad [e_1, e_2] = e_3.
\]
Consider the structure defined by $Je_1 = e_2$, $Je_3 = e_1 + e_4$, $Je_4 = -e_2 -e_3$. Then $\V \cap \g' = \langle e_3 \rangle$ and we get a $J$-bracket-adapted frame 
\[
\langle e_3, e_1 + e_4, e_2, -e_1 \rangle.
\]
Note that the torsion bundle is bracket-generating.

$\bullet \, \g_{4,2}^\alpha$: the structure equations are
\[
[e_1, e_4] = \alpha e_1, \quad [e_2, e_4] = e_2, \quad [e_3, e_4] = e_2 + e_3 \quad \text{with $\alpha \neq 0$.}
\]
Consider the structure defined by $Je_1 = e_3$, $Je_2 = e_4$. Then $\V \cap \g' = \langle e_2 + (1-\alpha) e_3 \rangle$ and, for $\alpha \neq 1$, we get a $J$-bracket-adapted frame 
\[
\langle e_2 + (1-\alpha) e_3, e_4 - (1-\alpha) e_1, (2 - \alpha) e_2 + (1-\alpha) e_3, (2 - \alpha) e_4 - (1-\alpha) e_1 \rangle.
\]
For the structure defined by $Je_1 = e_3 + e_4$, $Je_3 = e_2 - e_1$, $Je_2 = e_4 $, we have that $\V \cap \g' = \langle (\alpha -1)^2 e_1 + e_2 + (1-\alpha) e_3 \rangle$. In this case, the structure is bracket-generating since 
\[
\V = \langle (\alpha -1)^2 e_1 + e_2 + (1-\alpha) e_3, (\alpha-1) e_1 + e_4\rangle.
\]
For $\alpha =1$, the Lie algebra $\g_{4,2}^1$ does not admit almost complex structures with non-degenerate torsion bundle, see Lemma \ref{lemma:g42}.

$\bullet \, \g_{4,4}$: the structure equations are
\[
[e_1, e_4] = e_1, \quad [e_2, e_4] = e_1 + e_2, \quad [e_3, e_4] = e_2 + e_3.
\]
Consider the structure defined by $Je_1 = e_3$, $Je_2 = e_4$. Then $\V \cap \g' = \langle e_2\rangle$ and we get a $J$-bracket-adapted frame 
\[
\langle e_2, e_4, e_1 + e_2, e_3 + e_4 \rangle.
\]
For the structure defined by $Je_1 = e_3 + e_4$, $Je_3 = e_2 - e_1$, $Je_2 = e_4 $, we have that $\V \cap \g' = \langle e_1 + e_2 + e_3 \rangle$. In this case, the structure is bracket-generating since 
\[
\V = \langle e_1 + e_2 + e_3, -e_1 +e_2 +e_3 + 2 e_4\rangle.
\]

$\bullet \, \g_{4,5}^{\alpha, \beta}$: the structure equations are
\[
[e_1, e_4] = e_1, \quad [e_2, e_4] = \beta e_2, \quad [e_3, e_4] = \alpha e_3,
\]
where the parameters satisfy the conditions $-1 < \alpha \le \beta \le 1$ and $\alpha \beta \neq 0$, or the conditions $\alpha = -1$ and $0 < \beta \le 1$. Consider the structure defined by $Je_1 = e_4$, $Je_2 = e_1 + e_3 + e_4$, $Je_3 = e_1 - e_2 - e_4$. If the conditions $\alpha \neq 1$, $\beta \neq 1$, $\alpha \neq \beta$ are met, one can check that,  $\V \cap \g' = \langle X \rangle$, where
\begin{align*}
X &= N_J ((\alpha-1)e_1 + (\beta-1)e_4, e_2) \\
& = 2 (\alpha-1)(\beta-1)e_1 + (\beta-1)(\beta-\alpha)e_2 + (\alpha-1)(\beta-\alpha)e_3.
\end{align*}
We can directly compute that
\begin{align*}
T = [X, JX] &=  C [2 (\alpha-1)(\beta-1)e_1 + (\beta-1)(\beta-\alpha)e_2 + (\alpha-1)(\beta-\alpha)e_3, e_4]\\
& = C(2 (\alpha-1)(\beta-1)e_1 + \beta (\beta-1)(\beta-\alpha)e_2 + \alpha(\alpha-1)(\beta-\alpha)e_3),
\end{align*}
where $C = (\alpha-1)^2 + (\beta -1)^2 \neq 0 $. For the condition on the parameters, $T$ is not in the image of the Nijenhuis tensor, and the corresponding $J$-bracket adapted frame is
\[
\langle X, JX, T, JT\rangle.
\]
In this case, the torsion bundle is  also bracket-generating since the bracket $[T, JX]$ provides a vector in $\g'$ that is independent from $X$ and $T$.

For $\alpha=1$, $\beta =1$, or $\alpha= \beta$, the Lie algebra $\g_{4,5}^{\alpha,\beta}$ does not admit almost complex structures with non-degenerate torsion bundle, see Lemma \ref{lemma:g42}.

$\bullet \, \g_{4,6}^{\alpha, \beta}$: the structure equations are
\[
[e_1, e_4] = \alpha e_1, \quad [e_2, e_4] = \beta e_2 - e_3, \quad [e_3, e_4] = e_2 + \beta e_3 \quad \text{with $\alpha >0$ and $\beta \in \R$.}
\]
Consider the structure defined by $Je_1 = e_3$, $Je_2 = e_4$. Then $\V \cap \g' = \langle e_2 + (\beta - \alpha ) e_3\rangle$ and we get a $J$-bracket-adapted frame 
\[
\langle e_2 + (\beta-\alpha)e_3, e_4 - (\beta-\alpha)e_1, (2 \beta - \alpha)e_2 + (\beta^2 -\alpha \beta -1) e_3, (2 \beta - \alpha)e_4 - (\beta^2 -\alpha \beta -1) e_1  \rangle,
\]
which is not, however, bracket-generating. If we consider the structure defined by $Je_1 = e_2 + e_3 + e_4$, $Je_3 = -e_1 + e_2 - e_4$, $Je_2 = e_4 $, we have that
\[
\V = \langle X, JX \rangle,
\]
where $X= (1 + (\beta -\alpha)^2)e_1 + (3\beta - 3\alpha +1) e_2 + (\beta -\alpha-3) e_3$, and the torsion bundle is bracket-generating.

$\bullet \, \g_{4,7}$: the structure equations are
\[
[e_1, e_4] = 2 e_1, \quad [e_2, e_4] = e_2, \quad [e_3, e_4] = e_2 + e_3, \quad [e_2, e_3] = e_1.
\]
Consider the structure defined by $Je_1 = e_4$, $Je_2 = e_1 + e_3 + e_4$, $Je_3 = e_1 - e_2 - e_4$. Then $\V \cap \g' = \langle 2e_1 + 2e_2 - e_3\rangle$ and we get a $J$-bracket-adapted frame 
\[
\langle 2e_1 + 2e_2 - e_3, e_1 + e_2 + 2e_3 + 5e_4, 25 e_1 + 5(e_2 - e_3), 5 (e_2 + e_3) + 35e_4 \rangle.
\]
For such a structure, the torsion bundle is bracket-generating. Observe that $\g_{4,7}$ admits also almost complex structures with $\dim_\R \V \cap \g' = 2$, like $Je_1 = e_4$, $Je_2 = e_3$. In this case, a $J$-bracket-adapted frame is given by
\[
\langle e_2, e_3, e_1, e_4 \rangle. 
\]

$\bullet \, \g_{4,8}^{-1}$: the structure equations are
\[
[e_2, e_3] = e_1, \quad [e_2, e_4] = e_2, \quad [e_3, e_4] = - e_3.
\]
Consider the structure defined by $Je_1 = e_4$, $Je_2 = e_1 + e_2 + e_3$, $Je_3= - e_1 - 2 e_2 - e_3 - e_4$. Then $\V \cap \g' = \langle e_1 + 2 e_3\rangle$ and we get a $J$-bracket-adapted frame 
\[
\langle e_1 + 2 e_3, -2 e_1 -4 e_2 -2 e_3 -e_4, 8 e_1 + 2 e_3, -2 e_1 -4 e_2 -2 e_3 + 6 e_4 \rangle.
\]
However in this case $\V$ is not bracket-generating. For the structure defined by $Je_2 = e_1 + e_3 + e_4$, $Je_3 = e_1 - e_2 - e_4$, $Je_1 = e_4 $, we have that $\V \cap \g' = \langle e_1 - e_2 + e_3 \rangle$. In this case, the structure is bracket-generating since 
\[
\V = \langle e_1 - e_2 + e_3, e_2 + e_3 + e_4\rangle.
\]
Observe that $\g_{4,8}^{-1}$ admits also almost complex structures with $\dim_\R \V \cap \g' = 2$, like $Je_1 = e_4$, $Je_2 = e_3$. In this case, a $J$-bracket-adapted frame is given by
\[
\langle e_2, e_3, e_1, e_4 \rangle. 
\]

$\bullet \, \g_{4,8}^{\alpha}$: the structure equations are
\[
[e_1, e_4] = (1+\alpha) e_1, \quad [e_2, e_4] = e_2, \quad [e_3, e_4] = \alpha e_3, \quad [e_2, e_3] = e_1 \quad \text{with $-1 < \alpha \le 1$.}
\]
Consider the structure defined by $Je_1 = e_2 + e_3$, $Je_2 = - e_1 - e_4$, $Je_3 = e_4$. For $\alpha \neq 1$ and $\alpha \neq 0$, we have that $\V \cap \g' = \langle e_3 + (1+\alpha)e_2 \rangle$, and we get a $J$-bracket-adapted frame
\[
\langle (1+\alpha) e_2 + 2 e_3, (1-\alpha) e_4 - (1+\alpha) e_1, (1- \alpha^2) e_2 + 2\alpha (1-\alpha) e_3, (\alpha^2 -1)e_1 - (\alpha-1)^2 e_4 \rangle.
\]
For such a structure, the torsion bundle is bracket-generating. If $\alpha=0$ the frame is only $J$-adapted, since in that case $\dim_\R \g' = 2$. We refer to Example \ref{ex:g'2} for a bracket-generating structure when $\alpha=0$. Observe that $\g_{4,8}^\alpha$ also admits almost complex structures with $\dim_\R \V \cap \g' = 2$, like $Je_1 = e_4$, $Je_2 = e_3$. In this case, a $J$-bracket-adapted frame is given by
\[
\langle e_2, e_3, e_1, e_4 \rangle. 
\]
For $\alpha =1$, the Lie algebra $\g_{4,8}^{1}$ does not admit almost complex structures with non-degenerate torsion bundle, see Lemma \ref{lemma:g42}.

$\bullet \, \g_{4,9}^{0}$: the structure equations are
\[
[e_2, e_3] = e_1, \quad [e_2, e_4] = - e_3, \quad [e_3, e_4] = e_2.
\]
Consider the structure defined by $Je_1 = e_3$, $Je_2 = e_4$. Then $\V \cap \g' = \langle e_2 + e_3 \rangle$ and we get a $J$-bracket-adapted frame 
\[
\langle e_2 + e_3, e_4 - e_1, e_2 - e_3, e_4 + e_1 \rangle,
\]
and the torsion bundle is bracket-generating.

$\bullet \, \g_{4,9}^{\alpha}$: the structure equations are
\[
[e_1, e_4] = 2\alpha e_1, \quad [e_2, e_4] = \alpha e_2 - e_3, \quad [e_3, e_4] = e_2 + \alpha e_3, \quad [e_2, e_3] = e_1 \quad \text{with $\alpha >0$.}
\]
Consider the structure defined by $Je_1 = e_3$, $Je_2 = e_4$. Then $\V \cap \g' = \langle e_2 + (1-\alpha) e_3 \rangle$ and we get a $J$-bracket-adapted frame 
\[
\langle e_2 + (1-\alpha) e_3, e_4 - (1-\alpha) e_1, e_2 - (1-\alpha + \alpha^2) e_3, e_4 + (1-\alpha + \alpha^2) e_1 \rangle,
\]
and the torsion bundle is bracket-generating.
\hfill $\blacksquare$
\end{example}

\newpage

\begin{table}[H]
\centering
\renewcommand{\arraystretch}{1.6}
\caption{Classification of Real Four-Dimensional Lie Algebras}
\label{tab:list}
\small
\makebox[\textwidth][c]{
\begin{tabular}{lllcc}
\toprule
\textbf{Label} & \textbf{Non-zero Brackets} & \textbf{Parameters/Notes} & $\dim \g'$ & $\dim \g^{(2)}$ \\ 
\midrule
$4\g_1$ & --- & Abelian & 0 & 0\\
\midrule
$\g_{2,1} \oplus 2\g_1$ & $[e_1, e_2] = e_1$ & $\mathfrak{aff}(\mathbb{R}) \oplus \mathbb{R}^2$ & 1 & 0\\
\midrule
$2\g_{2,1}$ & \makecell[l]{$[e_1, e_2] = e_1$ \\ $[e_3, e_4] = e_3$} & $\mathfrak{aff}(\mathbb{R}) \oplus \mathfrak{aff}(\mathbb{R})$ & 2 & 0\\
\midrule
$\g_{3,1} \oplus \g_1$ & $[e_2, e_3] = e_1$ & $\text{Heisenberg} \oplus \mathbb{R}$ & 1 & 0\\
\midrule
$\g_{3,2} \oplus \g_1$ & \makecell[l]{$[e_3, e_1] = e_1$ \\ $[e_2, e_3] = e_1 - e_2$} & & 2 & 0\\
\midrule
$\g_{3,3} \oplus \g_1$ & $[e_3, e_1] = e_1, [e_2, e_3] = -e_2$ & & 2 & 0\\
\midrule
$\g_{3,4}^0 \oplus \g_1$ & $[e_3, e_1] = -e_2, [e_2, e_3] = e_1$ & & 2 & 0\\
\midrule
$\g_{3,4}^\alpha \oplus \g_1$ & \makecell[l]{$[e_3, e_1] = \alpha e_1 - e_2$,\\ $[e_2, e_3] = e_1 - \alpha e_2$} & $\alpha > 0, \alpha \neq 1$ & 2 & 0\\
\midrule
$\g_{3,5}^0 \oplus \g_1$ & \makecell[l]{$[e_3, e_1] = e_2$ \\ $[e_2, e_3] = e_1$} & & 2 & 0\\
\midrule
$\g_{3,5}^\alpha \oplus \g_1$ & \makecell[l]{$[e_3, e_1] = \alpha e_1 + e_2$ \\ $[e_2, e_3] = e_1 - \alpha e_2$} & $\alpha > 0$ & 2 & 0\\
\midrule
$\g_{3,6} \oplus \g_1$ & \makecell[l]{$[e_1, e_2] = - e_3, [e_3, e_1] = e_2$ \\ $[e_2, e_3] = e_1$} & $\mathfrak{sl}(2, \mathbb{R}) \oplus \mathbb{R}$ & 3 & 3\\
\midrule
$\g_{3,7} \oplus \g_1$ & \makecell[l]{$[e_1, e_2] = e_3, [e_2, e_3] = e_1$ \\ $[e_3, e_1] = e_2$} & $\mathfrak{so}(3) \oplus \mathbb{R}$ & 3 & 3\\ 
\midrule
$\g_{4,1}$ & $[e_2, e_4] = e_1, [e_3, e_4] = e_2$ & & 2  & 0\\
\midrule
$\g_{4,2}^\alpha$ & \makecell[l]{$[e_1, e_4] = \alpha e_1, [e_2, e_4] = e_2$ \\ $[e_3, e_4] = e_2 + e_3$} & $\alpha \neq 0$ & 3 & 0\\
\midrule
$\g_{4,3}$ & $[e_1, e_4] = e_1, [e_3, e_4] = e_2$ & & 2 & 0\\
\midrule
$\g_{4,4}$ & \makecell[l]{$[e_1, e_4] = e_1, [e_2, e_4] = e_1 + e_2$ \\ $[e_3, e_4] = e_2 + e_3$} & & 3 & 0\\
\midrule
$\g_{4,5}^{\alpha, \beta}$ & \makecell[l]{$[e_1, e_4] = e_1, [e_2, e_4] = \beta e_2$ \\ $[e_3, e_4] = \alpha e_3$} & \makecell[l]{$-1 \leq \alpha \leq \beta \leq 1$,  $\alpha\beta \neq 0$\\ $\alpha = -1$, $0 < \beta \le 1$} & 3 & 0\\
\midrule
$\g_{4,6}^{\alpha, \beta}$ & \makecell[l]{$[e_1, e_4] = \alpha e_1, [e_2, e_4] = \beta e_2 - e_3$ \\ $[e_3, e_4] = e_2 + \beta e_3$} & $\alpha > 0, \beta \in \R$ & 3 & 0\\
\midrule
$\g_{4,7}$ & \makecell[l]{$[e_2, e_3] = e_1, [e_1, e_4] = 2e_1$ \\ $[e_2, e_4] = e_2, [e_3, e_4] = e_2 + e_3$} & & 3 & 1\\
\midrule
$\g_{4,8}^{-1}$ & \makecell[l]{$[e_2, e_3] = e_1, [e_2, e_4] = e_2$, \\ $[e_3, e_4] = - e_3$} & & 3 & 1\\
\midrule
$\g_{4,8}^\alpha$ & \makecell[l]{$[e_2, e_3] = e_1, [e_1, e_4] = (1+\alpha)e_1$ \\ $[e_2, e_4] = e_2, [e_3, e_4] = \alpha e_3$} & $-1 < \alpha \leq 1$ & \makecell[c]{2 (if $\alpha=0$) \\ 3 (if $\alpha \neq 0$)}  & \makecell[c]{0 (if $\alpha=0$) \\ 1 (if $\alpha \neq 0$)}\\
\midrule
$\g_{4,9}^0$ & \makecell[l]{$[e_2, e_3] = e_1, [e_2, e_4] = - e_3$, \\ $[e_3, e_4] = e_2$} & & 3 & 1\\
\midrule
$\g_{4,9}^\alpha$ & \makecell[l]{$[e_2, e_3] = e_1, [e_2, e_4] = \alpha e_2 - e_3$ \\ $[e_1, e_4] = 2\alpha e_1, [e_3, e_4] = e_2 + \alpha e_3$} & $\alpha > 0$ & 3 & 1\\
\midrule
$\g_{4,10}$ & \makecell[l]{$[e_1, e_3] = e_1, [e_2, e_3]=e_2$ \\ $[e_1, e_4] = -e_2, [e_2, e_4] = e_1$} & & 2 & 0\\ 
\bottomrule
\end{tabular}    
}
\end{table}

\section{Structures with non-degenerate torsion bundle on homogeneous manifolds}\label{sec:quotients}

In this section, we study connected homogeneous almost complex $4$-manifolds whose torsion bundle is non-degenerate by means of their $J$-adapted double cover and their automorphism group.
\vspace{.2cm}

We begin with a preliminary result on $J$-adapted double covers of homogeneous almost complex $4$-manifolds and their automorphism groups.

\begin{proposition}\label{prop:connectedness}
    Let $(M,J)$ be a connected homogeneous almost complex $4$-manifold with non-degenerate torsion bundle. Let $F$ be the $J$-adapted double cover of $(M,J)$ and let $G$ be the identity component of $\Aut (M,J)$. Then 
    \begin{itemize}
        \item [(i)] $F$ is connected if and only if $F \cong G$. In such a case, $G$ admits a subgroup $G_p$ of order $2$, which coincides with the isotropy subgroup of $p \in M$, and $M$ is a Lie group if and only if $G_p$ is central in $G$;

        \item [(ii)] $F$ is disconnected if and only if $M \cong G$. In such a case, $F \cong G \times \{0,1\}$.
    \end{itemize}
\end{proposition}

\begin{proof}
Since $M$ is connected and homogeneous, $M$ is diffeomorphic to a quotient $M\cong G/G_p$ with $G=\operatorname{Aut}_0(M,J)$ and $G_p$ the isotropy at $p$. By Theorem \ref{thm:auto:group}, $\dim G \le 4$. By construction $M \cong F/\mathbb Z_2$, the $\mathbb Z_2$ being the deck group of $\pi$ exchanging $f_p^{+}$ and $f_p^{-}$.

\textbf{(i)} If $F\cong G$, then $F$ is connected and $\pi\colon G\to M$ is a connected $2:1$ covering, hence regular; its deck group is the group of right translations by $G_p$, so $G_p$ is a subgroup of $G$ of order $2$. Conversely, suppose $F$ is connected. As $G=\operatorname{Aut}_0(M,J)=\operatorname{Aut}_0(F,J)$, by Lemma 3.9 in \cite{BM17}, the action of $G$ on $M\cong F/\mathbb Z_2$ lifts to $F$. Fix $f\in F$. The orbit $G\cdot f$ projects surjectively onto $M$ because $G$ is transitive on $M$. Since $\pi$ is a local diffeomorphism, $\dim(G\cdot f) = 4 =\dim F$, so the orbit is open. The orbits thus partition $F$ into open sets, and connectedness forces $G\cdot f=F$. Hence $F\cong G/H$ with $H$ the stabilizer of $f$, and since an automorphism fixing a $J$-adapted frame is the identity, $H=\{e\}$ and $F\cong G$. Finally $M=G/G_p$ is a Lie group iff $G_p$ is normal, i.e.\ central, being of order $2$.

\textbf{(ii)} If $F$ is disconnected, by Lemma \ref{lemma:parallelizable} it is the disjoint union of two sheets, each mapped pseudo-biholomorphically onto $M$. The lifted action  of $G$ preserves the sheets and is transitive on each with trivial stabilizer, so each sheet is diffeomorphic to $G$. Hence $M\cong G$ and $F\cong G\times\{0,1\}$. The converse, $M\cong G\Rightarrow F$ disconnected, follows from the fact that $G$ is parallelizable and the construction of a $J$-adapted frame is global on $M$.
\end{proof}

The distinction between case (i) and (ii) depends on whether or not a certain involution of $J$-adapted frames is an inner automorphism.

\begin{proposition}\label{prop:inner}
Let $(M,J)$ be a connected homogeneous almost complex $4$-manifold with non-degenerate torsion bundle, let $(F,J)$ be its $J$-adapted double cover, and let $\psi$ be the involution of $J$-adapted frames $\psi(\{X,JX,T,JT\})=\{-X,-JX,T,JT\}$. 

By Proposition \ref{prop:connectedness}, $F$ is a Lie group with
identity component $G=\operatorname{Aut}_0(M,J)$. Let $h\in F$ be the element of order $2$ generating the deck group, so that $\operatorname{Ad}_h=\psi$. Then $h$ is
non-central, and exactly one of the following holds.
\begin{enumerate}
\item[(i)] $h\in G$. Then $F=G$ is connected, $\pi\colon G\to M$ is a $2:1$ covering, and $M=G/\langle h\rangle$ is the quotient of a connected Lie group by a non-central subgroup of order $2$. In particular $M$ is not a Lie group.

\item[(ii)] $h\notin G$. Then $F=G\sqcup hG$ is disconnected, $\pi|_G\colon G\to M$ is a diffeomorphism, and $M\cong G$.
\end{enumerate}
Moreover, $\operatorname{Ad}_h\notin\operatorname{Inn}(\mathfrak g)$ implies \emph{(ii)}.
\end{proposition}

\begin{proof}
By Proposition \ref{prop:connectedness}, the cover $F$ is a Lie group with identity component $G$, and $h^2=e$ with $\operatorname{Ad}_h=\psi$. As $\psi \neq \operatorname{id}$, $h\notin Z(F)$. Being a double cover of the connected $M$, $F$ has at most two components. Hence either $h\in G$, in which case $F=G$, or $h\notin G$, in which case $F=G\sqcup hG$.

(i) If $h\in G$, the order-two subgroup $\langle h\rangle$ acts freely on $G$ by right translation and $\pi\colon G\to G/\langle h\rangle=M$ is a $2:1$ covering. Since $h$ is non-central, $\langle h\rangle$ is not normal, so $M$ carries no group structure compatible with $\pi$.

(ii) If $h\notin G$, then for $g_1,g_2\in G$ one has $\pi(g_1)=\pi(g_2)$, that is equivalent to $g_2\in\{g_1,g_1h\}$. As $g_1h\in hG$ lies outside of $G$, this forces $g_2=g_1$, so $\pi|_G$ is injective. It is surjective because every coset $\{ g, gh \}$ meets $G$. Hence $\pi|_G$ is a diffeomorphism and $M\cong G$.

To prove the last assertion, note that if $h\in G$ then $\operatorname{Ad}_h\in\operatorname{Ad}(G)=\operatorname{Inn}(\mathfrak g)$.
\end{proof}

To prove our main result, we will resort multiple times to the following Lemma.

\begin{lemma}\label{lemma:absorption}
Let $Q$ be a smooth manifold, let $\phi \in \operatorname{Diff}_0(Q)$ be isotopic to the identity, and let $k\ge 1$ such that $\phi^{k} = \operatorname{Id}_{Q}$. Let
$\zeta \in \mathbb Z_{k}$ act on $Q \times S^{1}$ by
\[
   \zeta\cdot(p,\theta)=\bigl(\phi(p),\,\theta + \tfrac\pi k\bigr).
\]
Then the action is free and $(Q \times S^{1})/\mathbb Z_{k}$ is diffeomorphic to $Q \times S^{1}$.
\end{lemma}

\begin{proof}
The translation $\theta\mapsto\theta+\tfrac \pi k$ has no periodic point of period $j$ for $0<j<k$, so $\zeta^{j}$ acts without fixed points for such $j$, the action is free, and the quotient is a smooth manifold. Projection onto the second factor realizes $(Q \times S^{1})/\mathbb Z_{k}$ as a fiber bundle over $S^{1}/\mathbb Z_{k}$ with fiber $Q$, which is the mapping torus $T_{\phi}$ of $\phi$. Mapping tori of isotopic diffeomorphisms are diffeomorphic, and $T_{\operatorname{Id}_{Q}} = Q \times S^{1}$. Since $\phi \in \operatorname{Diff}_0 (Q)$, we conclude that $T_{\phi}$ is diffeomorphic to $Q \times S^{1}$.
\end{proof}

We are ready to state our main result.

\begin{theorem}\label{thm:diffeotype}
    Let $(M,J)$ be a connected homogeneous almost complex $4$-manifold with non-degenerate torsion bundle. Then $M$ is diffeomorphic to a $4$-dimensional Lie group admitting an almost complex structure with non-degenerate torsion bundle, or to the products $L(4,1) \times \R$, $L(4,1) \times \T$, where $L(4,1)$ is a lens space.
\end{theorem}
\begin{proof}

Let $(M,J)$ be a connected homogeneous almost complex $4$-manifold with non‑degenerate torsion bundle, and let $\pi\colon F\to M$ be its $J$-adapted double cover. By Proposition \ref{prop:connectedness}, we have two possibilities.

\textbf{Case 1: $M$ is a Lie group.}
Then $M$ carries a left‑invariant almost complex structure with non‑degenerate torsion bundle. Its universal cover $\tilde M$ inherits a left‑invariant structure of the same type, so the Lie algebra of $\tilde M$ must be one of those classified in Theorem \ref{thm:classification}. The explicit examples in Section \ref{sec:algebras}, Examples \ref{ex:g'2} and \ref{ex:g'3}, show that all admissible Lie algebras do occur. Moreover, the explicit left-invariant structures provided descend from $\tilde{M}$ to $M$, establishing the first alternative of the theorem.

\textbf{Case 2: $M$ is not a Lie group.}
By Proposition \ref{prop:connectedness}, we have that $M\cong G/H$ where $G=\operatorname{Aut}_0(M,J)$ is a $4$-dimensional connected Lie group and $H=\{e,h\}$ is a non‑central subgroup of order $2$ acting on $G$ by right multiplication. The complete list of pairs $(G,H)$, up to conjugacy, is given in the Appendix (Propositions \ref{prop:no:subgroups}--\ref{prop:sl2}). We now determine the diffeomorphism type of $G/H$ for each such pair, dividing them in two groups.

\textbf{Quotient $M = G /H$ diffeomorphic to $G$.} In all the pairs below, $G$ is diffeomorphic to a product $ Q \times S^{1}$. The action of $h$ shifts the $S^{1}$ factor by a half-period and acts on $Q$ by a map isotopic to the identity. Hence, Lemma \ref{lemma:absorption} applies and yields $G/H \cong_{C^{\infty}} G$. Each such $G$ is a $4$-dimensional Lie group admitting an almost complex structure with non-degenerate torsion bundle, so these pairs fall under the first alternative of the theorem.

\begin{itemize}
\item For $G=SE_n(2)\times\mathbb R$ with $n$ odd (Proposition \ref{prop:SE2}), the generator $h=g_n$ shifts the rotation circle of the $SE_n(2)$ factor by a half-period and rotates the $\mathbb R^{2}$ part by $\pi$, acting trivially on the $\mathbb R$ factor. Taking $ Q=\mathbb R^{2}\times\mathbb R$ and $\phi$ the $\pi$-rotation of $\mathbb R^2$ (times the identity), Lemma \ref{lemma:absorption} with $k=2$ gives $G/H\cong_{C^{\infty}}G$. For $G=SE_n(2)\times\mathbb T$ the same holds for both subgroups: the generator $g_n^{\pi}$ additionally shifts the $\mathbb T$ factor by a half-period, which is again isotopic to the identity, so $\phi$ stays isotopic to $\operatorname{Id}_Q$.

\item For the quotients of $G^0_{4,9}$ (Proposition \ref{prop:G9}) and of $G_{4,10}$ (Proposition \ref{prop:G10}) with $n$ odd, the involution $h$ shifts the unique circle direction by a half-period (the $e_4$, resp.\ $z$, coordinate) and acts on the transverse $\mathbb R^{2}$ by a rotation, composed with a translation in the
remaining directions. Both factors of this transverse action are isotopic to the identity, so $\phi \simeq \operatorname{Id}$ and Lemma \ref{lemma:absorption} with $k=2$ gives $G/H\cong_{C^{\infty}}G$.

\item For the quotients of $\widetilde{SL}(2,\mathbb R)\times\mathbb R$ (Proposition \ref{prop:sl2}) with $n$ odd, the generator $h$ shifts the circle produced by the central lattice by a half-period and acts on the transverse plane $\mathfrak p=\langle e_1,e_2\rangle$ by the rotation of angle $n\pi$, composed with a translation in the $e_4$-direction. As above $\phi \simeq \operatorname{Id}$, and $G/H\cong_{C^{\infty}}G$. The same conclusion holds for the quotients with a $\T$ factor, where $h$ may shift the torus factor by a half-period.

\item For $G=SO(3)\times\mathbb T$ and the diagonal subgroup
$H=\{ \Id,(r,e^{i\pi})\}$ of Proposition \ref{prop:SU2}, with $r\in SO(3)$ the non-central element of order $2$, the generator shifts the $\mathbb T$ factor by a half-period and acts on $SO(3)$ by right translation $R_r$, which is isotopic to the identity.
Lemma \ref{lemma:absorption} with $Q = SO(3)$ and $k=2$ gives
$(SO(3)\times\mathbb T)/H\cong_{C^{\infty}}SO(3)\times\mathbb T$.

\item For $G=U(2)$ (Proposition \ref{prop:SU2}), write $U(2)=(SU(2)\times S^{1})/\mathbb Z_2$ with $\mathbb Z_2=\langle\delta\rangle$, and $\delta(u,\theta)=(-u,\theta+\tfrac12)$. Then Lemma \ref{lemma:absorption} with $Q=SU(2)$, $\phi=R_{-\mathrm {Id}}$ the antipodal map, and $k=2$ already gives
$U(2)\cong_{C^{\infty}}SU(2)\times S^{1}=S^{3}\times S^{1}$. The non-central
$h=\operatorname{diag}(1,-1)$ has $\det h=-1$, so it lifts to
$\hat h(u,\theta)=(uc,\theta+\tfrac14)$ with $c=\operatorname{diag}(-i,i)$ of order
$4$ and $\hat h^{2}=\delta$. Thus $\langle\delta,\hat h\rangle=\langle\hat h\rangle
\cong\mathbb Z_4$ and
\[
   U(2)/H=(SU(2)\times S^{1})/\mathbb Z_4 .
\]
Since right multiplication by $c$ is isotopic to the identity in $SU(2)$, Lemma \ref{lemma:absorption} applied with $k=4$ gives $U(2)/H\cong_{C^{\infty}}SU(2)\times S^{1}=S^{3}\times S^{1}$. Hence $U(2)/H\cong_{C^{\infty}}U(2)$, a $4$-dimensional Lie group.
\end{itemize}

\textbf{Quotients diffeomorphic to lens spaces.} The remaining pairs are those in which $h$ acts on a compact $SO(3)$ factor with
no circle factor, and they produce the genuinely non-group diffeomorphism types of the theorem. Let $r\in SO(3)$ be the
non-central element of order $2$, realized as the rotation by $\pi$ about a fixed axis. Lifting $r$ to $SU(2)=S^3$ yields a quaternion of order $4$, so $\langle r\rangle$ lifts to a cyclic group $\mathbb Z_4$ acting freely on $S^{3}$
by right translation, whence
\[
   SO(3)/\langle r\rangle\;\cong_{C^{\infty}}\;S^{3}/\mathbb Z_4\;=\;L(4,1).
\]
For $G=SO(3)\times\mathbb R$, the unique non-central subgroup of order $2$ is $H=\{e,(r,0)\}$, giving $M\cong L(4,1)\times\mathbb R$. For $G=SO(3)\times\mathbb T$ and the subgroup $H=\{e,(r,1)\}$ acting trivially on $\mathbb T$, we obtain $M\cong L(4,1)\times\mathbb T$. An explicit almost complex structure with non-degenerate torsion bundle on $L(4,1)\times\mathbb R$ and $L(4,1)\times\mathbb T$ is constructed in Example \ref{ex:lens}. This is the second alternative of the theorem.
\end{proof}

\begin{example}[\textbf{Almost complex structure with non-degenerate torsion bundle on }$L(4,1)\times\mathbb{R}$ \textbf{and} $L(4,1)\times\mathbb{T}$]\label{ex:lens}

Consider the Lie algebra $\mathfrak{so}(3)\oplus\mathbb{R}
=\langle e_1,e_2,e_3,e_4\rangle$ with $[e_1,e_2]=e_3$, $[e_2,e_3]=e_1$, $[e_3,e_1]=e_2$, and the left-invariant almost complex structure $J$ defined by
\[
  Je_1=e_4,\qquad Je_2=e_2+e_3,\qquad Je_3=-2 e_2-e_3.
\]
Since $J$ is left-invariant and the projection
$SU(2)\to SO(3) \times \R$ is a quotient by a central subgroup, $J$ descends to a left-invariant structure of the same type on $SO(3)\times\mathbb{R}$, and on $SO(3)\times\mathbb{T}$.

Let $h\in SO(3)$ be the non-central order-$2$ element of Proposition \ref{prop:SU2}, realized as the rotation by $\pi$ about the $e_1$-axis, so that
\[
  \operatorname{Ad}_h e_1=e_1,\quad \operatorname{Ad}_h e_2=-e_2,\quad
  \operatorname{Ad}_h e_3=-e_3,\quad \operatorname{Ad}_h e_4=e_4 .
\]
A $J$-adapted frame is given by $\{X,JX,T,JT\}$ with $X=3 e_2 + e_3$, $JX=e_2+2 e_3$, $T=[X,JX]= 5 e_1$, and $JT= 5 e_4$. On this frame $\operatorname{Ad}_h$ acts as $\psi(\{X,JX,T,JT\})=\{-X,-JX,T,JT\}$ of Proposition~\ref{prop:inner}. Thus $\operatorname{Ad}_h$ commutes with $J$, so right multiplication by $h$ is pseudoholomorphic and $J$ descends to the quotient $\big(SO(3)\times\mathbb{R}\big)/H\cong L(4,1)\times\mathbb{R}$, where $H=\{e,h\}$. The same construction on $SO(3)\times\mathbb{T}$ yields a structure of the same type on $L(4,1)\times\mathbb{T}$.
\hfill  $\blacksquare$
\end{example}

\appendix

\section{Non-central subgroups of order \texorpdfstring{$2$}{} of \texorpdfstring{$4$}{}-dimensional Lie groups}\label{sec:appendix}

In this section we determine non-central subgroups of order $2$ of $4$-dimensional connected Lie groups. As a starting point, we base ourselves on the classification of $4$-dimensional Lie groups made by Biggs and Remsing, adopting their notation \cite{BR16}. In some cases, we will recur to the following 

\begin{lemma}\label{lemma:quotient}
    Let $G$ be a Lie group and let $H$ be a discrete central subgroup of $G$ that is torsion‑free. Suppose that $\{ \id, g \}$ is a non-central subgroup of $G$ of order $2$.
    Then $\{ [\id], [g] \}$ is a non-central subgroup of $G/H$ of order $2$.
\end{lemma}
\begin{proof}
    Since $H$ is central and $g\notin Z(G)$, we have that $g\notin H$, hence $[g]\neq[\id]$, and $\{[\id],[g]\}$ is a non-trivial subgroup of $G/H$ of order $2$.
    
    Assume, by contradiction, that $[g]$ is central in $G/H$.
    Then for every $x\in G$ we have $[g][x]=[x][g]$, i.e.\ $gxg^{-1}x^{-1}\in H$.
    Denote $c_x = gxg^{-1}x^{-1}\in H$. Since $g$ has order $2$, we can apply conjugation by $g$ twice to $x$, to get \[
    g ( g x g^{-1}) g^{-1} = x.
    \]
    On the other hand, we also have that 
    \[
    g(gxg^{-1})g^{-1 }= g(c_x x)g^{-1} = c_x gxg^{-1} =c_x^2 x,
    \]
    where we used that $c_x$ is central, that gives us the equality $c_x^2 x = x$. Since $c_x\in H$ and $H$ is torsion‑free, $c_x=\id$ for every $x\in G$. Thus $gxg^{-1}=x$ for all $x$, i.e.\ $g\in Z(G)$, contradicting the hypothesis.
\end{proof}

For the sake of exposition, we divide the Lie groups in several classes.

\subsection{Groups without non-central subgroups of order \texorpdfstring{$2$}{}.}\label{sec:nosubgroups}

\begin{proposition}\label{prop:no:subgroups}
    The following $4$-dimensional Lie groups have no non-central subgroups of order $2$:
    \begin{itemize}
        \item [(i)] Abelian groups $\R^k \times \T^l$, $k+l=4$;
        \item [(ii)] trivial Abelian extensions of the Lie groups $G_{2,1}$, $G_{3,2}$, $G_{3,3}$, $G^0_{3,4}$, $G^\alpha_{3,4}$, $G^\alpha_{3,5}$ and their group quotients;
        \item [(iii)] the Lie group $G_{2,1} \times G_{2,1}$;
        \item [(iv)] Lie groups whose universal cover is the trivial Abelian extension of the Heisenberg group;
        \item [(v)] the Lie groups $G^\alpha_{4,2}$, $G_{4,4}$, $G^{\alpha,\beta}_{4,5}$, $G^{\alpha,\beta}_{4,6}$, $G_{4,7}$, $G^\alpha_{4,8}$, $G^\alpha_{4,9}$;
        \item [(vi)] Lie groups whose universal cover is either $G_{4,1}$, $G_{4,3}$, $G^{-1}_{4,8}$.
    \end{itemize}
\end{proposition}
\begin{proof}
    All the Lie groups mentioned in (i--iii) and (v) are linearizable, hence it is a standard computation to check that either they do not have subgroups of order $2$ or that such subgroups are central, cf.\ Sections 3.1, 3.2, 3.3 and 4.1 in \cite{BR16}. For cases (iv) and (vi), the reasoning is as follows: the universal covers $H_3 \times \R$, where $H_3$ is the $3$-dimensional Heisenberg group, $G_{4,1}$, $G_{4,3}$ and $G^{-1}_{4,8}$ are linearizable, and it can be directly verified that they have no subgroups of order $2$. However, they have non-trivial center and some of their quotients might have (and actually, they do have) subgroups of order $2$. To rule out this possibility, we have to study more in depth their diagram of cover. We show here the explicit computations for $H_3 \times \R$, which presents the most complicated diagram. The other cases are completely analogous, using the explicit matrix representations and cover diagrams given in \cite{BR16}.
    
    Explicitly, $H_3 \times \R$ is represented as
    \[
    H_3 \times \R \cong \left\{  
    \begin{bmatrix}
        1 & x & w & 0 \\
        0 & 1 & y & 0 \\
        0 & 0 & 1 & 0 \\
        0 & 0 & 0 & e^z 
    \end{bmatrix}
    : w,x,y,z \in \R
    \right\}.
    \]
    Fix a basis $\{ e_j \}_{j=1}^4$ of its Lie algebra $\g$ such that
    \[
    \g = \{ w e_1 + x e_2 + y e_3 + z e_4 \}.
    \]
    Every connected Lie group whose universal cover is $H_3 \times \R$ fits in the diagram of covers

    {\small
    \begin{center}
        \begin{adjustbox}{center}
            \begin{tikzcd}
         & H_3 \times \R \arrow[ld] \arrow[rd] & \\
        H_3 \times \R / exp( \Z e_1) \arrow[rd] & & H_3 \times \R / exp( \Z e_4) \arrow[ld] \arrow[dd] \\
         & H_3 \times \R / exp( \Z e_1)exp(\Z e_4) & \\
         & & H_3 \times \R / exp( \Z (e_1 + \alpha e_4))exp(\Z e_4)
            \end{tikzcd}
        \end{adjustbox}
    \end{center}
    }
    where $exp$ denotes the exponential map and $\alpha \in (0,1) \setminus \Q$ is a real parameter. Let $\{[\id],[g]\}$ be a subgroup of order $2$ of $H_3 \times \R / exp( \Z e_1)exp(\Z e_4)$. Then $[g]^2 = [\id]$, which implies that $g^2 \in  exp( \Z e_1) \, exp(\Z e_4)$. Since $H_3 \times \R$ is linearizable, we can directly check that $g \in Z(H_3 \times \R)$, hence $\{[\id],[g]\}$ is central. The same argument applies to the quotient $H_3 \times \R / exp( \Z (e_1 + \alpha e_4))exp(\Z e_4))$. By Lemma \ref{lemma:quotient}, any non‑central subgroup of order $2$ in a group from the diagram would project to a non‑central subgroup of order $2$ in the corresponding maximal quotient, contradicting the computation above.
\end{proof}

\subsection{Lie groups with the same universal cover as \texorpdfstring{$\mathrm{SE}(2)\times \R$}{}.}\label{sec:se2}
Let $\mathrm{SE}(2)$ be the Euclidean group with universal cover $\widetilde{\mathrm{SE}}(2)$
\[
\widetilde{\mathrm{SE}}(2) =
\left\{ 
\begin{bmatrix}
    1 & 0 & 0 & 0\\
    x & \cos z & - \sin z & 0\\
    y & \sin z & \cos z & 0 \\
    0 & 0 & 0 & e^z
\end{bmatrix}
: x,y,z \in \R 
\right\}.
\]
Every Lie group whose universal cover is $\widetilde{\mathrm{SE}}(2) \times \R$ fits into the following diagram, cf.\ Section 3.5 in \cite{BR16}:
{\small
    \begin{center}
    \begin{tikzcd}
    \widetilde{\mathrm{SE}}(2) \times \R \arrow[r] \arrow[d] & \widetilde{\mathrm{SE}}(2) \times \T \arrow[d] \\
    \mathrm{SE}_n(2) \times \R \arrow[r] \arrow[d] & \mathrm{SE}_n(2) \times \T \arrow[d] \\
    \mathrm{SE} (2) \times \R \arrow[r] & \mathrm{SE}(2) \times \T \\
    \end{tikzcd}
    \end{center}
}
where $\mathrm{SE}_n(2)$ its the $n$-fold cover of $\mathrm{SE}(2)$. 
\begin{proposition}\label{prop:SE2}
    The universal cover $\widetilde{\mathrm{SE}}(2) \times \R$ has no non-central subgroups of order $2$. The Lie group $\mathrm{SE}_n(2) \times \R$, $n \ge 1$, has non-central subgroups of order $2$ if and only if $n$ is odd. In such a case, the subgroup is unique, up to conjugacy, and is given by $\{ [\id], [g_n] \}$, where
    \[
    g_n = \begin{bmatrix}
        1 & 0 & 0 & 0 \\
        0 & -1 & 0 & 0 \\
        0 & 0 & -1 & 0 \\
        0 & 0 & 0 & e^{n\pi} \\
    \end{bmatrix} \times \{ 0 \},
    \]
\end{proposition}

\begin{proof}
    Denote by $exp$ the exponential map and by $e_3$ the element in the Lie algebra of $\widetilde{\mathrm{SE}}(2) \times \R$ corresponding to the $z$ variable. The Lie group $\mathrm{SE}_n (2) \times \R$ is obtained as the quotient $\widetilde{\mathrm{SE}}(2) \times \R / ( H_n \times \{ 0 \})$, where $H_n = exp( 2 \pi n \Z e_3)$, $n \ge 0$. Note that the case $ n =0$ corresponds, at least formally, to $\widetilde{\mathrm{SE}}(2) \times \R $. An element $g \in \widetilde{\mathrm{SE}}(2) \times \R$ defines a subgroup of order $2$ in $\mathrm{SE}_n (2) \times \R$ if and only if $g^2 \in H_n \times \{ 0\}$. Since the $\R$-factor commutes with elements of $\widetilde{\mathrm{SE}}(2)$, we can treat it separately, that will give $\pi_\R (g) = 0$. Writing explicitly the equation $g^2 \in H_n$, we get the matrix equality
    \begin{gather*}
    \begin{bmatrix}
        1 & 0 & 0 & 0 \\
        x (1 + \cos z) - y \sin z & \cos 2z& - \sin 2z & 0 \\
        y ( 1+ \cos z) + x \sin z & \sin 2z & \cos 2z  & 0 \\
        0 & 0 & 0 & e^{2z} 
    \end{bmatrix} = \\
    \begin{bmatrix}
        1 & 0 & 0 & 0 \\
        0 & \cos 2\pi n k & - \sin 2 \pi nk & 0 \\
        0 & \sin 2\pi n k & \cos 2 \pi nk & 0 \\
        0 & 0 & 0 & e^{2 \pi nk}
    \end{bmatrix}
    \end{gather*}   
    for some $k \in \Z$. If $nk$ is even, then necessarily $ z \in 2 \pi \Z$, $x=y=0$, and the subgroup defined by $[g]$ is central. If $nk$ is odd, then
    \[
    g_{nk} (a,b) = \begin{bmatrix}
        1 & 0 & 0 & 0 \\
        a & -1 & 0 & 0 \\
        b & 0 & -1 & 0 \\
        0 & 0 & 0 & e^{nk\pi} \\
    \end{bmatrix} \times \{ 0 \},
    \]
    with $a$, $b \in \R$, defines a non-central subgroup of order $2$ in the quotient $\mathrm{SE}_n (2) \times \R$. Moreover, up to conjugacy by an element of $\mathrm{SE}_n (2) \times \R$, we can assume $a=b=0$. Passing to the  quotient, we get a unique class $[g_n]$, as in the statement of the proposition.
\end{proof}

Taking quotients by translations in the $\R$-factor, one gets a similar result for $\widetilde{\mathrm{SE}}(2) \times \T$ and $\mathrm{SE}_n(2) \times \T$. In this case there are two non-central subgroups of order $2$, namely $\{ [\id], [g_n] \}$ and $\{ [\id], [g_n^\pi] \}$, where
\[
 g_n^\pi = \begin{bmatrix}
        1 & 0 & 0 & 0 \\
        0 & -1 & 0 & 0 \\
        0 & 0 & -1 & 0 \\
        0 & 0 & 0 & e^{n\pi} \\
    \end{bmatrix} \times \{ e^{i \pi} \}.
\]

\subsection{Lie groups whose universal cover is \texorpdfstring{$G^0_{4,9}$}{} or \texorpdfstring{$G_{4,10}$}{}.}\label{sec:G9} 
The computations for the Lie groups $G^0_{4,9}$ and $G_{4,10}$ are very similar to those we performed for $\widetilde{\mathrm{SE}}(2)\times \R$. Hence, we just state our result, explicitly writing the non-central subgroup of order $2$ and omitting the proof.\\

The Lie group $G^0_{4,9}$ is linearizable and can be realized as
\[
G^0_{4,9} \cong \left\{
\begin{bmatrix}
    1& -x \cos z - y \sin z & y \cos z - x \sin z & -2 w & 0 \\
    0 & \cos z & \sin z & y & 0 \\
    0 & - \sin z & \cos z & x & 0 \\
    0 & 0 & 0 & 1 & 0 \\
    0 & 0 & 0 & 0 & e^z
\end{bmatrix}: w,x,y,z \in \R
\right\}.
\]
Quotients of $G^0_{4,9}$ fit into the diagram 
    {\small
        \begin{center}
        \begin{tikzcd}
        G^0_{4,9} \arrow[r] \arrow[d] & G^0_{4,9}/ exp (\Z e_1) \arrow[d] \\
        G^0_{4,9}/ exp (2 n \pi \Z e_4) \arrow[r] \arrow[d] & G^0_{4,9}/ exp (\Z e_1) \, exp (2 n \pi \Z e_4)\arrow[d] \\
        G^0_{4,9}/ exp (2 \pi \Z e_4) \arrow[r] & G^0_{4,9}/ exp (\Z e_1) \, exp (2 \pi \Z e_4) \\
        \end{tikzcd}
        \end{center}
    }
Their non-central subgroups of order $2$ are described in the next proposition.

\begin{proposition}\label{prop:G9}
    The universal cover $G^0_{4,9}$ and the quotient $G^0_{4,9}/ exp (\Z e_1)$ have no non-central subgroups of order $2$. The quotient $G^0_{4,9}/ exp (2 \pi n \Z e_4)$ has a unique, up to conjugacy, non-central subgroup of order $2$, given by $\{ [\id], [g_n] \}$, where
    \[
    g_n = \begin{bmatrix}
        1 & 0 & 0 & 0 & 0 \\
        0 & -1 & 0 & 0 & 0 \\
        0 & 0 & -1 & 0 & 0 \\
        0 & 0 & 0 & 1 & 0 \\
        0 & 0 & 0 & 0 & e^{n\pi}
    \end{bmatrix}.
    \]
    For $G^0_{4,9}/ exp (\Z e_1) \, exp (2 \pi n \Z e_4)$, we have two non-central subgroups of order $2$, generated by the class of
    \[
    g_n^\sigma = \begin{bmatrix}
        1 & 0 & 0 & \sigma & 0 \\
        0 & -1 & 0 & 0 & 0 \\
        0 & 0 & -1 & 0 & 0 \\
        0 & 0 & 0 & 1 & 0 \\
        0 & 0 & 0 & 0 & e^{n \pi}
    \end{bmatrix}\quad \text{with $\sigma \in \{ 0 , -1 \}$}.
    \]
\end{proposition}
\vspace{.2cm}

The Lie group $G_{4,10}$ is linearizable and can be realized as
\[
G_{4,10} \cong \left\{
\begin{bmatrix}
    e^{-y} \cos z & e^{-y} \sin z & x & 0 \\
    -e^{-y} \sin z & e^{-y} \cos z & w & 0 \\
    0 & 0 & 1 & 0 \\
    0 & 0 & 0 & e^z
\end{bmatrix}: w,x,y,z \in \R
\right\}.
\]
Its non-central subgroups of order $2$ are described in the next proposition.

\begin{proposition}\label{prop:G10}
    The universal cover $G_{4,10}$ has no non-central subgroups of order $2$. Quotients of $G_{4,10}$ are parametrized by $n \in \Z$ and each of them has non-central subgroups of order $2$ if and only if $n$ is odd. In that case, the subgroup is unique, up to conjugacy, and is given by $\{ [\id], [g_n] \}$, where
    \[
    g_n = \begin{bmatrix}
        -1 & 0 & 0 & 0 \\
        0 & -1 & 0 & 0 \\
        0 & 0 & 1 & 0 \\
        0 & 0 & 0 & e^{n \pi}
    \end{bmatrix}.
    \]
\end{proposition}

\subsection{Lie groups whose universal cover is \texorpdfstring{$\mathrm{SU}(2)\times \R$}{}.} Let $\mathrm{SU}(2)$ be the group of unitary $2 \times 2$ matrices with determinant equal to $1$, that is
\[
\mathrm{SU}(2) = \{ g \in \C^{2 \times 2}: g^* g = \Id, \, \, \det g = 1 \}.
\]
Every connected Lie group whose universal cover is $\mathrm{SU}(2) \times \R$ fits in the following diagram of covers
{\small
    \begin{center}
    \begin{tikzcd}
     & \mathrm{SU}(2) \times \R \arrow[ld] \arrow[d] \arrow[rd] & \\
    \mathrm{SO}(3) \times \R  \arrow[rd] & \mathrm{SU}(2) \times \T \arrow[d] & \mathrm{U}(2) \arrow[ld] \\
     & \mathrm{SO}(3) \times \T & \\
    \end{tikzcd}
    \end{center}
}
where $\mathrm{SO}(3)$ is the group of special orthogonal matrices
\[
\mathrm{SO}(3) = \{ g \in \R^{3 \times 3}: g^t g = \Id, \, \, \det g = 1 \}
\]
and $\mathrm{U}(2)$ is the group of unitary matrices
\[
\mathrm{U}(2) = \{ g \in \C^{2 \times 2}: g^* g = \Id \}.
\]

\begin{proposition}\label{prop:SU2}\hfill
\begin{itemize}
    \item [(i)] $\mathrm{SU}(2) \times \R$ and $\mathrm{SU}(2) \times \T$ have no non-central subgroups of order $2$.

    \item [(ii)] $\mathrm{U}(2)$ has a unique, up to conjugacy, non-central subgroup of order $2$, given by
    \[
    \left\{ \Id, 
    \begin{bmatrix}
        1 & 0 \\
        0 & -1 \\
    \end{bmatrix} \right\}.
    \]

    \item [(iii)] $\mathrm{SO}(3)$ has a unique, up to conjugacy, non-central subgroup of order $2$, given by
    \[
    \left\{ \Id, 
    \begin{bmatrix}
        1 & 0 & 0\\
        0 & -1 & 0\\
        0 & 0 & -1
    \end{bmatrix} \right\}.
    \]
    Hence, $\mathrm{SO}(3) \times \R$ and $\mathrm{SO}(3) \times \T$ have one and two non-central subgroups of order $2$, respectively.
\end{itemize}
\end{proposition}

\begin{proof}
    Suppose that $g \in \mathrm{SU}(2)$ has order $2$. Then by definition we have $g^{-1} = g^*$. Moreover, $g^{-1} = g$ since $g^2 = \Id$. Hence $g$ is a self-adjoint matrix and it is diagonalizable via conjugacy by a unitary matrix. Since $g^2 = \Id$, its only eigenvalues are $\pm 1$. This, together with the condition $\det g = 1$, implies that either $g= \Id $, which is excluded, or $g = -\Id$, which provides a central subgroup. The same is true for the quotient $\mathrm{SU}(2) \times \T$, where every subgroup of order $2$ is central. This proves (i).\\
    The proof of (ii) is very similar. Let $g \in \mathrm{U}(2)$ be an element of order $2$. Then $g = g^*$, $g$ can be diagonalized over $\C$ and its eigenvalues are $\pm 1$. If they have the same sign, we obtain a central subgroup. Hence the only possibility is that $g$ is obtained from  
    \[ 
    \begin{bmatrix}
        1 & 0 \\
        0 & -1 
    \end{bmatrix}
    \]
    by conjugacy. The proof of (iii) is the same as (ii), replacing unitary with orthogonal, to study subgroups of $\mathrm{SO}(3)$. The statement follows after taking product with $\R$, which has no element of order $2$, or $\T$, which has $e^{i\pi}$ as the only element of order $2$.
\end{proof}

\subsection{Lie groups with the same universal cover as \texorpdfstring{$\mathrm{SL}(2, \R)\times \R$}{}.}\label{sec:sl2} Let $\mathrm{SL}(2, \R)$ be the special linear group. Denote by $\Tilde{A}$ its universal cover. Let $\{ e_j \}_{j=1}^4$ be a basis of the Lie algebra of $\Tilde{A} \times \R$. Then $\Tilde{A} \times \R$ is diffeomorphic to $\R^4$ via the map
\begin{equation}\label{eq:diffeo}
\begin{split}
\R^4 &\longrightarrow \Tilde{A} \times \R \\
(w,x,y,z) &\longmapsto exp (y e_3) exp (w e_1 + x e_2) exp ( z e_4).
\end{split}
\end{equation}
Discrete central subgroups of $\Tilde{A} \times \R$ are classified in \cite{BR16}. We adopt the following notation
\begin{align*}
    &H_n^\alpha = exp ( n \Z (2 \pi e_3 + \alpha e_4)) exp (n  \Z e_4) \\
    &H_n^1 = exp ( n \Z (2 \pi e_3 + e_4))\\
    &H_n = exp ( 2n\pi \Z e_3 )
\end{align*}
Every Lie group whose universal cover is $\Tilde{A} \times \R$ fits in the following diagram of covers
{\small
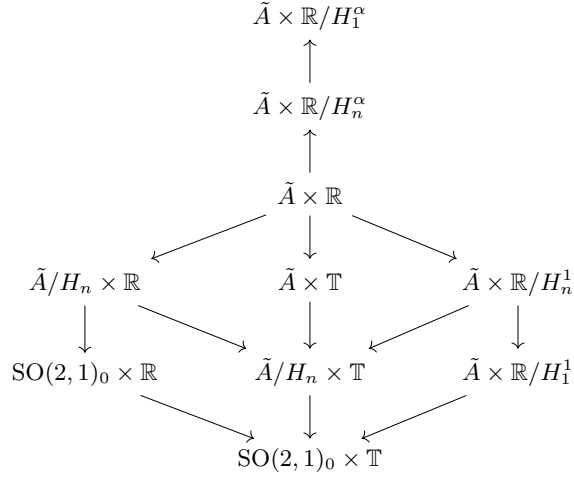
    \begin{center}
    \begin{tikzcd}
     & \Tilde{A} \times \R / H^\alpha_1 & \\
     & \Tilde{A} \times \R / H^\alpha_n \arrow[u] & \\
     & \Tilde{A} \times \R \arrow[ld] \arrow[d] \arrow[rd] \arrow[u] & \\
    \Tilde{A} / H_n \times \R \arrow[rd] \arrow[d] & \Tilde{A} \times \T \arrow[d] & \Tilde{A} \times \R  / H^1_n \arrow[d] \arrow[ld] \\
     \mathrm{SO}(2,1)_0 \times \R \arrow[rd] & \Tilde{A} / H_n \times \T \arrow[d] & \Tilde{A} \times \R / H^1_1 \arrow[ld] \\
     & \mathrm{SO}(2,1)_0 \times \T & \\
    \end{tikzcd}
    \end{center}
}

\begin{proposition}\label{prop:sl2}
    The universal cover $\Tilde{A} \times \R$ has no non-central subgroups of order $2$. The quotients $\Tilde{A} / H_n \times \R$, $\Tilde{A} \times \R / H^1_n $, and $\Tilde{A} \times \R / H^\alpha_n$, all have non-central subgroups of order $2$ if and only if $n$ is odd. Such a subgroup is unique, up to conjugacy, and it is given by $[g]$, where $g = exp (n \pi e_3)$, $g = exp (n \pi e_3 + \frac{n}{2} e_4)$, and $g = exp (n \pi e_3) \, exp ( \frac{n}{2} (\alpha + 1) e_4)$, respectively.
\end{proposition}
\begin{proof}
    Since $\widetilde A\times\R$ is simply connected and the maximal compact subgroup of $A$, namely $SO(2)$, lifts to $\R$, the group $\widetilde A\times\R$ is torsion-free. Hence it has no subgroups of order $2$, and we only examine the central quotients $G_\Gamma=(\widetilde A\times\R)/\Gamma$, with $\Gamma\in\{H_n,H^1_n,H^\alpha_n\}$ contained in $Z(\widetilde A\times\R)=\exp(\R e_3)\times\R$.
    
    Let $\tilde{a} = \mathfrak t \oplus \mathfrak p=\R\langle e_3\rangle\oplus\R\langle e_1,e_2\rangle$ be the Cartan decomposition, normalized so that $\operatorname{Ad}_{\exp(s e_3)}$ rotates $\mathfrak p$ by the angle $s$ and $\exp(2\pi e_3)$ generates $Z(SL(2,\R))$. Let $[g]\in G_\Gamma$ have order $2$ and write $g=\exp(X)\exp(t e_4)$, with $X=x_1e_3+P$, $P\in\mathfrak p$, $t\in\R$. Then $g^2\in\Gamma$, so the component of $g^2$ in $\tilde{A}$ lies in $\exp(\R e_3)$.
    
    The projection $\bar u$ of $\exp(X)$ on $SL(2,\R)$ satisfies $\bar u^2\in T=SO(2)$. In $SL(2,\R)$ the only elements whose square is in $SO(2)$ are conjugate of elements of $SO(2)$. Conjugacy classes lift along the central extension $\widetilde A\to SL(2,\R)$, so up to conjugacy $P=0$ and $g=\exp(y e_3)\exp(t e_4)$; since $\exp|_{\R e_3}$ is injective, $y$ is uniquely determined. From $g^2=\exp(2y e_3)\exp(2t e_4)\in\Gamma$, we can see that:
    
    \begin{itemize}
    
    \item $\Gamma=H_n$: the $e_4$-component is trivial, so $t=0$ and $2y\in 2n\pi\Z$, giving $y=n\pi k$ and $g=\exp(n\pi k\,e_3)$;
    
    \item $\Gamma=H^1_n$: the single generator forces $(2y,2t)=m\,n(2\pi,1)$, whence $m=k$, $t=nk/2$ and $g=\exp(n\pi k\,e_3+\tfrac{nk}{2}e_4)$;
    
    \item $\Gamma=H^\alpha_n$: the two generators give $(2y,2t)=k\,n(2\pi,\alpha)+b\,n(0,1)$, whence $t=\tfrac n2(k\alpha+b)$ and $g=\exp(n\pi k\,e_3)\exp(\tfrac n2(k\alpha+b)e_4)$, $b\in\Z$.
\end{itemize}

Since $\operatorname{Ad}_g$ rotates $\mathfrak p$ by $n\pi k$ and fixes $e_3,e_4$, the element $g$ is not central if and only if $nk$ is odd, in which case $k$ is odd. By considering $[g]$ in the quotient $G_\Gamma$, we can assume $k=1$
\end{proof}

After taking quotient by translations along the $\R$ factor, we see that non-central subgroups of order $2$ of $\Tilde{A} / H_n \times \T$ are obtained as products of those given in Proposition \ref{prop:sl2} with either $0$ or $\pi \in \T$.

{\small
\printbibliography
}

\end{document}